\documentclass[12pt]{amsart}
\usepackage{graphicx} 
\usepackage{xcolor}
\usepackage{adjustbox}
\usepackage{amsmath,amssymb,amsxtra,tikz,stackengine,bm,mathrsfs,bbm,tikz-cd,comment,hyperref,combelow,amsopn,float,anysize}
\usepackage{subcaption}
\usepackage{enumitem}
\usepackage{ytableau}
\usepackage{quiver}
\usetikzlibrary{calc}%
\usetikzlibrary{shapes}%
\usetikzlibrary{patterns}%
\usetikzlibrary{positioning}%
\usetikzlibrary{arrows.meta}
\usetikzlibrary{knots}
\usetikzlibrary{decorations.markings}
\usetikzlibrary{hobby}

\newcommand{\Hilb}{\mathrm{Hilb}}
\newcommand{\C}{\mathbb{C}}
\newcommand{\Z}{\mathbb{Z}}
\newcommand{\R}{\mathbb{R}}
\newcommand{\N}{\mathbb{N}}
\renewcommand{\P}{\mathbb{P}}
\newcommand{\CP}{\mathcal{P}}

\newcommand{\CL}{\mathcal{L}}
\newcommand{\CO}{\mathcal{O}}

\newcommand{\sgn}{\mathrm{sgn}}
\newcommand{\Sym}{\mathrm{Sym}}
\newcommand{\Ker}{\mathrm{Ker}}

\newcommand{\Spec}{\mathrm{Spec}}
\newcommand{\Pic}{\mathrm{Pic}}
\newcommand{\Proj}{\mathrm{Proj}}
\newcommand{\Hom}{\mathrm{Hom}}
\newcommand{\CV}{\mathcal{V}}
\newcommand{\HC}{\mathrm{HC}}

\DeclareMathOperator{\ch}{ch}

\newcommand{\aaa}{\mathbf{a}}
\newcommand{\bbb}{\mathbf{b}}

\newcommand{\Bl}{\mathrm{Bl}}
\newcommand{\BN}{\mathrm{BN}}

\numberwithin{equation}{section}

\newtheorem{theorem}{Theorem}[section]
\newtheorem{proposition}[theorem]{Proposition}
\newtheorem{corollary}[theorem]{Corollary}
\newtheorem{lemma}[theorem]{Lemma}

\theoremstyle{definition}
\newtheorem{remark}[theorem]{Remark}

\newtheorem{definition}[theorem]{Definition}
\newtheorem{example}[theorem]{Example}

\newcommand{\sq}{\square}

\title{Wall-crossing for Hilbert schemes}
\author{Ian Cavey}
\author{Eugene Gorsky}
\author{Alexei Oblomkov}
\author{Joshua P. Turner}

\begin{document}

\maketitle

\begin{abstract}
    The goal of the minimal model program for the Hilbert scheme of points on a surface aims is to describe the (stable) base loci of all divisors, their associated birational models, and the maps between them. We answer all of these questions for the Hilbert scheme of points on the blowup of the affine plane at the origin. The birational models are Brill-Noether loci in a larger Hilbert scheme, and nested variants thereof, and the wall-crossing maps are described as explicit projections. We also establish several new facts about the homogeneous coordinate ring of this Hilbert scheme, including finding the minimal set of line bundles whose sections generate the ring.
\end{abstract}

\section{Introduction}

The Hilbert scheme of $n$ points on a smooth, quasi-projective surface $X$, denoted $\Hilb^n(X)$, is a smooth, quasi-projective variety of dimension $2n$ parameterizing zero dimensional subschemes of $X$ of length $n$. The minimal model program for these Hilbert schemes, initiated in \cite{ABCH} for $X=\P^2$, aims to describe for every effective divisor $D$ on $\Hilb^n(X)$: the (stable) base locus of $D$, the birational model
\begin{equation} 
\Proj\left(\bigoplus_{m\geq 0} H^0(\Hilb^n(X),\CO(mD))\right),
\end{equation}
as well as the relationships between these models as $D$ varies in the effective cone of $\Hilb^n(X)$. In this paper, we carry out a local version of this program, running the minimal model program for the Hilbert scheme of points on $X=\Bl_\mathbf{0}(\C^2)$, the blowup of the affine plane at the origin.

Recall that $\Pic(\Hilb^n(\Bl_{\mathbf{0}}(\C^2)))\simeq \Z^2$ is generated by $\CO(E_n)$ and $\CO(\frac{1}{2}B)$ where $E_n\subseteq \Hilb^n(\Bl_{\mathbf{0}}(\C^2))$ is the divisor of schemes whose support meets the exceptional curve $E\subseteq \Bl_{\mathbf{0}}(\C^2)$, and $B\subseteq \Hilb^n(\Bl_{\mathbf{0}}(\C^2))$ is the divisor of nonreduced schemes \cite{Fogarty2}. To simplify notation, we write $\CO(k,m) := \CO_{\Hilb^n(\Bl_{\mathbf{0}}(\C^2))}(-kE_n-\frac{m}{2} B)$. We also have $\Pic(\Hilb^n(\C^2))\simeq \Z$ generated by $\CO(\frac{1}{2}B)$ where $B\subseteq \Hilb^n(\C^2)$ is defined similarly, and we write $\CO(m) = \CO_{\Hilb^n(\C^2)}(-\frac{m}{2}B)$.

Our results are based on an explicit description of the sections of line bundles on $\Hilb^n(X)$ for any smooth toric surface $X$ obtained in \cite{CGOT}, generalizing a result of Haiman \cite{Haimanqt} in the case $X= \C^2$. With this, we have identifications
\[ H^0(\Hilb^n(\Bl_{\mathbf{0}}(\C^2)),\CO(k,m)) = A^m_{\geq k}(n), \]
and 
\[ H^0(\Hilb^n(\C^2),\CO(m))= A^m(n), \]
where $A^m_{\geq k}(n),A^m(n)\subseteq \C[x_1,y_1,\dots,x_n,y_n]$ are certain subspaces of diagonally alternating or symmetric polynomials (see Section \ref{sec:background} for the precise definition).

Let $X_k = X_k(n)$ denote the closure of the image of the map \[ \iota_k: \Hilb^n(\C^2\setminus\mathbf{0})\to \Hilb^{n+k(k+1)/2}(\C^2) \]
defined by $I\mapsto (x,y)^k \cap I$. Geometrically, the image of $\iota_k$ is the set of subschemes that are disjoint unions of a multiplicity $k$ point at the origin with a length $n$ subscheme supported away from the origin, and $X_k(n)$ is the set of all subschemes that appear as limits of these. We show in Section \ref{sec:BN} that the $X_k$ are \emph{Brill-Noether loci} in the larger Hilbert scheme, however our present results do not make use of this fact.

Our first result shows that the varieties $X_k(n)$ appear as certain birational models of $\Hilb^n(\Bl_{\mathbf{0}}(\C^2))$ extending a result of Haiman for $X_0(n) = \Hilb^n(\C^2)$ \cite{Haimanqt}.

\begin{theorem}\label{thm:main1}
    For any $k\geq 0$, $X_k(n)$ is the birational model for $\Hilb^n(\Bl_{\mathbf{0}}(\C^2))$ corresponding to the line bundle $\CO(k,1)$, i.e.
    \[ X_k(n) \simeq \Proj\left( \bigoplus_{m\geq 0}A^m_{\geq km}(n)\right). \]
\end{theorem}

Theorem \ref{thm:main1} is proved in Section \ref{sec:proof of main thm}. To prove Theorem \ref{thm:main1}, we describe the restriction of sections of line bundles to $X_k(n)\subseteq \Hilb^{n+k(k+1)/2}(\C^2)$ as a ring homomorphism
\[ \bigoplus_{m\geq 0} A^m\left( n+\frac{k(k+1)}{2} \right) \xrightarrow{\iota_k^*} \bigoplus_{m\geq 0} A^m\left( n \right). \] 
We then show that the image of this map, the homogeneous coordinate ring of $X_k(n)$, is
\[ \bigoplus_{m\geq 0}A^m_{\geq km}(n) \]
as desired by explicit computation in degrees $m=0,1$ combined with new finite-generation facts about these rings described below.

Similarly, let $X_{k,k+1} = X_{k,k+1}(n)$ be the closure of the image of the product map 
\[ \iota_k\times\iota_{k+1}:\Hilb^n(\C^2\setminus\mathbf{0})\to \Hilb^{n+k(k+1)/2}(\C^2)\times \Hilb^{n+(k+1)(k+2)/2}(\C^2). \]
For $k=0$ we have $X_{0,1}=\Hilb^{n,n+1}_0(\C^2)$
parametrizing pairs of ideals $(I,J)$ such that $\C[x,y]\supset I\supset J$ and $I/J$ is supported at the origin.

The space $X_{k,k+1}(n)$ has a natural bigraded section ring coming of the restrictions of products of line bundles $\CO(m_1)$ and $\CO(m_2)$, $m_1,m_2\geq 0$, from each factor. Using a similar strategy to that of Theorem \ref{thm:main1} we show the following.

\begin{theorem}\label{thm:main2}
    For any $k\geq 0$, the bigraded section ring of $X_{k,k+1}(n)$ is isomorphic to 
    \[ \bigoplus_{m_1,m_2\geq 0} A^{m_1+m_2}_{\geq km_1+(k+1)m_2}(n). \]
    In particular, for any $m_1,m_2>0$ $X_{k,k+1}(n)$ is the birational model for $\Hilb^n(\Bl_{\mathbf{0}}(\C^2))$ corresponding to the line bundle \[ \CO(km_1+(k+1)m_2,m_1+m_2) = \CO(k,1)^{\otimes m_1}\otimes \CO(k+1,1)^{\otimes m_2}. \]
\end{theorem}

These models give a concrete description of the wall-crossing phenomenon as an alternating sequence of natural projection maps, where the rays spanned by the line bundles $\CO(k,1)$ for $k\geq 0$ corresponding to the $X_k$'s, and the chambers between them correspond to the $X_{k,k+1}$'s:
\begin{equation}\label{eq:mapdiagram}
    \begin{tikzcd}
    &X_{0,1} \arrow[dl,"p_{0}"'] \arrow[dr,"q_1"] && X_{1,2} \arrow[dl,"p_1"'] \arrow[dr,"q_{2}"] && X_{2,3} \arrow[dl,"p_{2}"']\arrow[dr,"q_{3}"] \\
    X_0 &&X_1 && X_{2} && \cdots
\end{tikzcd}
\end{equation} 

Sequences of the form \eqref{eq:mapdiagram} have appeared in the work of Bayer, Chen, and Jiang \cite{Bayer} on Brill-Noether loci in Hilbert schemes of points on surfaces, as well as the work of Nakajima and Yoshioka \cite{NY2} on more general moduli spaces of sheaves on blow-ups of projective surfaces. 

Recalling that the ample cone of $\Hilb^n(\Bl_\mathbf{0}(\C^2))$ is the interior of the cone generated by $\CO(n-1,1)$ and $\CO(1,0)$, our description of the (nested) Brill-Noether loci as birational models of $\Hilb^n(\Bl_{\mathbf{0}}(\C^2))$ immediately gives the following stability property for \eqref{eq:mapdiagram}.

\begin{corollary}\label{cor:X_k and blowup}
    For $k\geq n$, we have $X_k(n)\simeq X_{k-1,k}(n)\simeq \Hilb^n(\Bl_{\mathbf{0}}(\C^2)).$
\end{corollary}
Indeed, Theorems \ref{thm:main1} and \ref{thm:main2} realize $X_k(n)$ and $X_{k,k+1}(n)$ as birational models of the variety $\Hilb^n(\Bl_{\mathbf{0}}(\C^2))$ whose ample cone is the interior of the cone generated by $\CO(n-1,1)$ and $\CO(1,0)$. The birational models associated to ample line bundles are isomorphic to $\Hilb^n(\Bl_{\mathbf{0}}(\C^2))$, which gives the corollary.

Corollary \ref{cor:X_k and blowup}  reduces the infinite collection of walls and chambers given by Theorems \ref{thm:main1} and \ref{thm:main2} to a finite collection for the purposes of classifying the birational models of $\Hilb^n(\Bl_{\mathbf{0}}(\C^2))$ up to isomorphism leaving only the rays spanned by $$\CO(1,1),\CO(2,1),\dots,\CO(n-1,1),$$ at least in the range $k,m\geq 0$. For example, Figure \ref{fig:Chambers for n=4} shows this decomposition for $\Hilb^4(\Bl_{\mathbf{0}}(\C^2))$ labeled with the corresponding birational model for each wall and chamber.

\begin{figure}[h]
    \centering
    \begin{tikzpicture}
        \draw[<-,thick] (8,0)--(0,0);
        \draw[->,thick] (0,0)--(0,3);
        \draw[->,thick] (0,0) -- (3,3);
        \draw[->,thick] (0,0) -- (6,3);    
        \draw[->,thick] (0,0) -- (8,2.67);   
        \node[left] at (0,2.5) {$m$};
        \node[below] at (7.5,0) {$k$};
        \foreach \x in {0,1,...,7}
            \foreach \y in {0,1,...,2}{
                \draw (\x,\y) circle (3pt);
            };
        \filldraw (0,0) circle (3pt);
        \filldraw (0,1) circle (3pt);
        \filldraw (1,1) circle (3pt);
        \filldraw (2,1) circle (3pt);
        \filldraw (3,1) circle (3pt);
        \filldraw (1,0) circle (3pt);
        \node at (1.4,2.8) {$\Hilb^{4,5}_0(\C^2)$};
        \node at (4,2.8) {$X_{1,2}$};
        \node at (7,2.8) {$X_{2,3}$};
        \node at (9,1.5) {$\Hilb^4(\Bl_{\mathbf{0}}(\C^2))$};
        \node at (0,3.5) {$\Hilb^4(\C^2)$};
        \node at (3.3,3.3) {$X_1$};
        \node at (6.4,3.3) {$X_2$};
        \node at (8.5,2.9) {$X_3$};
        \node at (9.5,0) {$\Sym^4(\Bl_{\mathbf{0}}(\C^2))$};
    \end{tikzpicture}
    \caption{Walls and chambers in the Picard group and the corresponding birational models for $\Hilb^4(\Bl_{\mathbf{0}}(\C^2))$.}
    \label{fig:Chambers for n=4}
\end{figure}
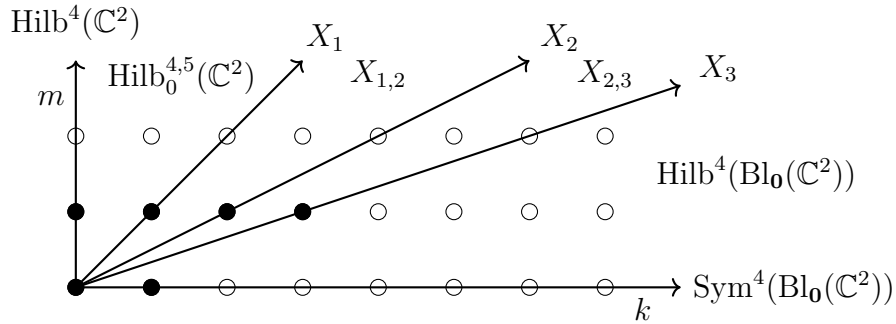

The following theorem, proved in Section \ref{sec:base locus}, describes of the base locus of these line bundles on $\Hilb^n(\Bl_\mathbf{0}(\C^2))$.

\begin{theorem}\label{thm:base locus decomp}
    For any $k\geq 0$ and $m>0$, the base locus of $\CO(k,m)$ is the locus of subschemes $Z\in \Hilb^n(\Bl_\mathbf{0}(\C^2))$ such that $|Z\cap E|>k/m$, where $|Z\cap E|$ denotes the length of the scheme-theoretic intersection of $Z$ with the exceptional curve $E\subseteq \Bl_{\mathbf{0}}(\C^2)$.
\end{theorem}

To carry the coordinate ring computations needed for the proofs of Theorems \ref{thm:main1} and \ref{thm:main2}, we first establish several facts about the total homogeneous coordinate ring of $\Hilb^n(\Bl_\mathbf{0}(\C^2))$,
\[ \bigoplus_{k,m\in \Z} H^0(\Hilb^n(\Bl_\mathbf{0}(\C^2)),\CO(k,m)) = \bigoplus_{k,m\in \Z} A^m_{\geq k}(n). \]
We summarize these results in the following theorem. Whenever we refer to the generators of a ring of sections, we mean generators as an algebra over the global functions $A^0$ of $\Hilb^n(\Bl_\mathbf{0}(\C^2))$.

\begin{theorem}\label{thm:finitegenerationfacts}
    The following hold:
    \begin{enumerate}
        \item The total homogeneous coordinate ring of $\Hilb^n(\Bl_\mathbf{0}(\C^2))$ is minimally generated by the sections of 
        \[ \CO(0,1),\CO(1,1),\dots,\CO(n-1,1), \text{ and } \CO(1,0),\]
        plus the canonical sections $s_{E_n}$ and $s_B$ of $\CO(-1,0)$ and $\CO(0,-2)$ respectively.
        \item For any two consecutive line bundles  $\CL_1$ and $\CL_2$ from the list $\CO(0,1),$ $\CO(1,1),$ $\dots,$ $\CO(n-1,1)$, and $\CO(1,0)$, the bigraded ring of sections 
        \[ \bigoplus_{d_1,d_2\geq 0} H^0(\Hilb^n(\Bl_\mathbf{0}(\C^2)),\CL_1^{\otimes d_1}\otimes \CL_2^{\otimes d_2}) \]
        is generated by the sections of $\CL_1$ and $\CL_2$.
        \item For any $k,m\geq 0$ the ring of sections 
        \[ \bigoplus_{d\geq 0} H^0(\Hilb^n(\Bl_\mathbf{0}(\C^2)),\CO(k,m)^{\otimes d}) \]
        is generated by the sections of $\CO(k,m)$.
    \end{enumerate}
\end{theorem}

We prove these algebraic statements combinatorially using the methods developed in \cite{CGOT}. Specifically, we identified the sets $\mathcal{P}(k,m)\subseteq \Z^{2n}_{\geq 0}$ of trailing term exponents of elements of $A^m_{\geq k}$ with respect to a certain term order. In Section \ref{sec:division}, we give several additivity properties about these sets from which we deduce Theorem \ref{thm:finitegenerationfacts}.

This combinatorial approach allows us to characterize \emph{Newton-Okounkov bodies} of the models $X_k$ and $X_{k,k+1}(n)$, extending our previous results for $\Hilb^{n,n+1}_0(\C^2)$ and $\Hilb^n(\Bl_\mathbf{0}(\C^2))$ \cite{CGOT}. This application is discussed in Section \ref{sec:NObodies}. We also show in this section the walls and chambers determining the birational models and base loci of line bundles on $\Hilb^n(\Bl_\mathbf{0}(\C^2))$ described above coincides with the regions on which these Newton-Okounkov bodies vary linearly. And moreover, we give in Proposition \ref{prop:base locus analogue for NO body} a combinatorial analogue of our description of base loci of line bundles on $\Hilb^n(\Bl_\mathbf{0}(\C^2))$. This combinatorial model was useful in guessing the geometric statements in Theorems \ref{thm:main1}, \ref{thm:main2}, and \ref{thm:base locus decomp}. 

Finally, in Section \ref{sec: character} we give a combinatorial formula for the $(\C^{\times})^2$-equivariant holomorphic Euler characteristics of ample line bundles on the varieties $X_{k,k+1}$ as a weighted sum over $\mathcal{P}(k,m)$. The sets $\mathcal{P}(k,m)$ can be partitioned into subsets on which the weighted sum is an easily computable rational function (see Example \ref{ex:weightedcount}). On the other hand, these rational functions can be computed using the equivariant localization formula, and comparing these two expressions gives nontrivial identities.

\section*{Acknowledgments}

The work of E. G. was partially supported by the NSF grant DMS-2302305. The work of J. P. T. was supported in part by the Pacific Institute for the Mathematical Sciences.

\section{Background}\label{sec:background}

\subsection{Hilbert schemes}

Let $\sgn(m)=\sgn^{\otimes m}$ denote the one-dimensional representation of $S_n$ which is trivial for $m$ even and sign for $m$ odd. Similarly, if $V$ is a representation of $S_n$ then the $\sgn(m)$ component of $V$ is $V^{S_n}$ if $m$ is even and $V^{\sgn}$ if $m$ is odd. Consider the diagonal action of $S_n$ on $\C[x_1,\ldots,x_n,y_1,\ldots,y_n]$ which permutes $x_i, y_i$ simultaneously. We define the bigrading on these polynomial rings by $\deg(x_i)=\deg(x)=q$ and $\deg(y_i)=\deg(y)=t$, and note that the action of $S_n$ preserves the grading. We will simply refer to $\deg$ as to degree.

\begin{definition}
Let $A\subset \C[x_1,\ldots,x_n,y_1,\ldots,y_n]$ be the subspace of antisymmetric polynomials.
\end{definition}

\begin{definition}
\label{def: delta S}
Let $S\subseteq \Z^2_{\ge 0}$ be an $n$-element subset. We will always assume without loss of generality that the elements of $S$ are labeled in increasing lexicographic order: $S=\{(a_1,b_1),\ldots,(a_n,b_n)\}$, and define $\Delta_{S}=\det(x_i^{a_j}y_i^{b_j})$. 
\end{definition}

Note that a different choice of ordering for elements of $S$ yields the same $\Delta_S$ up to a sign. One can think of $\Delta_S$ as the antisymmetrization of a monomial $x_1^{a_1}y_1^{b_1}\cdots x_n^{a_n}y_n^{b_n}$, which implies that $\Delta_S$ form a homogeneous basis of $A$. 

\begin{definition}
We define the integer powers of $A$ as follows. If $m\le 0$ then $A^m$ is the $\sgn(m)$ component of $\C[x_1,\ldots,x_n,y_1,\ldots,y_n]$. If $m>0$ then $A^m$ is the span of products of $m$-tuples of elements of $A$.
\end{definition}

In particular, $A^0=\C[x_1,\ldots,x_n,y_1,\ldots,y_n]^{S_n}$. 

\begin{theorem}\cite{Haiman,Haimanqt}\label{thm:secC2HS}
For all $m\in \Z$ one has
$$
H^0(\Hilb^n(\C^2),\CO(m))=A^m.
$$
\end{theorem}

\begin{definition}
    For integers $k,m$, we define $A^m_{\geq k}\subseteq A^m$ to be the subset of polynomials for which every term $x_1^{a_1}y_1^{b_1}\cdots x_n^{a_n}y_n^{b_n}$ has $a_i+b_i\geq k$ for all $i$.
\end{definition}

In a previous work, we established the following generalization of Theorem \ref{thm:secC2HS}.

\begin{theorem}\cite{CGOT}
    For all $k,m\in \Z$ one has
    \[ H^0(\Hilb^n(\Bl_{\mathbf{0}})(\C^2),\CO(k,m)) = A^m_{\geq k}. \]
\end{theorem}

Our main tool in this paper is a combinatorially-indexed $\C$-linear basis of $A^m_{\geq k}$, which we now recall.

\begin{definition}\label{def:P(k,m)}
    For any integers $k,m\geq0$, let $\mathcal{P}(k,m)\subseteq\N^{2n}$ be the set of points $(a_1,b_1,\dots,a_n,b_n)$ such that:
    \begin{enumerate}
        \item $0\leq a_1\leq a_2 \leq \cdots \leq a_n$,
        \item for any $j=1,\dots,n-1$ for which $a_j = a_{j+1}$, we have $b_{j+1}\geq b_j+m$, and
        \item for each $j=1,\dots,n,$ we have $b_j\geq \max\{k-a_j,0\} + \sum_{i=1}^{j-1}\max\{m-(a_j-a_i),0\}.$
    \end{enumerate}
\end{definition}

Note that $(a_1,b_1,\dots,a_n,b_n)\in \mathcal{P}(k,1)$ if and only if  $(a_i,b_i)\in \N^2$ are a collection of distinct points with $a_i+b_i\geq k$ and such that
\[ (a_1,b_1)<\cdots<(a_n,b_n) \]
in lexicographic order. Similarly, $(a_1,b_1,\dots,a_n,b_n)\in\mathcal{P}(k,0)$ if and only if the $(a_i,b_i)\in \N^2$ are a collection of not-necessarily distinct points with $a_i+b_i\geq k$ and such that
\[ (a_1,b_1)\leq\cdots\leq(a_n,b_n) \]
in lexicographic order.

For the following result and throughout this paper, we equip $\C[x_1,y_1,\dots,x_n,y_n]$ with the lexicographic term order where the variables are ordered $x_1>\cdots>x_n>y_1>\cdots>y_n$.

\begin{theorem}\label{thm:HS lead terms}\cite{CGOT}
    For all $k,m\in \Z$, there exists a polynomial in $A^m_{\geq k}$ with trailing term $x_1^{a_1}y_1^{b_1}\cdots x_n^{a_n}y_n^{b_n}$ if and only if $(a_1,b_1,\dots,a_n,b_n)\in \mathcal{P}(k,m)$.
\end{theorem}

\subsection{Brill-Noether loci}\label{sec:BN}

In this section we introduce the main varieties of interest.

\begin{definition}
    Let $X_k$ denote the closure of the image of the map
    \[\iota_{k}: \Hilb^n(\C^2\setminus \mathbf{0}) \to \Hilb^{n+k(k+1)/2}(\C^2)\]
    defined by $\iota_{k}:I\mapsto (x,y)^k\cap I$, and $X_{k,k+1}$ the closure of the image of the product map
    \[ \iota_k\times \iota_{k+1}: \Hilb^n(\C^2\setminus \mathbf{0}) \to \Hilb^{n+k(k+1)/2}(\C^2)\times \Hilb^{n+(k+1)(k+2)/2}(\C^2). \]
\end{definition}

By construction, the spaces $X_k$ and $X_{k,k+1}$ are irreducible varieties of dimension $2n$. We have natural projections
\begin{equation}
\label{eq: def p q maps}
\begin{tikzcd}
 & X_{k,k+1} \arrow[swap]{dl}{p_k} \arrow{dr}{q_{k+1}} & \\
 X_k & & X_{k+1}.
\end{tikzcd}
\end{equation}
It turns out that $X_k$ and $X_{k,k+1}$ are closely related to the Brill-Noether loci.

\begin{definition}
Define the Brill-Noether loci
$$
\BN^{r}(n)=\{I\subset \C[x,y]:\dim I/(x,y)I\ge r+1\}\subset \Hilb^n(\C^2),
$$
and
$$
\mathcal{M}(n,r):=\left\{(I,J):\C[x,y]\supset I\supset J,\ xI\subset J,\ yI\subset J\right\}\subset \Hilb^{n,n+r+1}(\C^2).
$$
\end{definition}

It is easy to see that $\BN^r(n)$ is closed in $\Hilb^n(\C^2)$ and that the image of the natural projection 
$$
p^+_{n,r}:\mathcal{M}(n,r)\to \Hilb^n(\C^2)
$$
is contained in $\BN^r(n)$. Indeed, for $(I,J)\in \mathcal{M}(n,r)$ we have $I\supset J\supset (x,y)I$ and $\dim I/(x,y)I\ge \dim I/J=r+1$. We can also consider the second projection
$$
p^-_{n,r}:\mathcal{M}(n,r)\to \Hilb^{n+r+1}(\C^2).
$$
For $(I,J)\in \mathcal{M}(n,r)$ the $\C[x,y]$-module $\C[x,y]/J$ contains an $(r+1)$-dimensional trivial submodule $I/J$, and one can check   \cite[Lemma 3.2]{Bayer} that this implies $\dim J/(x,y)J\ge r+2$. Therefore the image of $p^-_{n,r}$ is contained in $\BN^{r+1}(n+r+1)$, and we get the diagram
\begin{equation}
\label{eq: def BN maps}
\begin{tikzcd}
 & \mathcal{M}(n,r) \arrow[swap]{dl}{p_{n,r}^{+}} \arrow{dr}{p_{n,r}^-} & \\
 \BN^r(n) & & \BN^{r+1}(n+r+1).
\end{tikzcd}
\end{equation}
By the main result of \cite{Bayer}, the locus
$\BN^r(n)$ is irreducible and 
$$
\dim \BN^r(n)=2n-r(r+1).
$$
Furthermore, by \cite[Proposition 3.23]{NY2} (see also \cite{GGS}) the space $\mathcal{M}(n,r)$ is smooth  and the maps $p^+_{n,r}$ and $p^-_{n,r}$ are resolutions of singularities of $\BN^r(n)$ and $\BN^{r+1}(n+r+1)$ respectively.

\begin{proposition}
\label{prop: X as BN}
We have isomorphisms $$X_k=\BN^k(n+k(k+1)/2),\ X_{k,k+1}=\mathcal{M}(n+k(k+1)/2,k)$$
and the diagrams \eqref{eq: def p q maps} and \eqref{eq: def BN maps} agree. In particular, $X_{k,k+1}$ is smooth and the projections $p_k:X_{k,k+1}\to X_k$ and $q_k:X_{k,k+1}\to X_{k+1}$ are proper and birational.
\end{proposition}

\begin{proof}
It is easy to see that the image of $\iota_k$ is contained in $\BN^k(n+k(k+1)/2)$, hence $X_k\subset \BN^k(n+k(k+1)/2)$. Since both $X_k$ and $\BN^k(n+k(k+1)/2)$ are irreducible varieties of dimension $2n$, we get $X_k=\BN^k(n+k(k+1)/2)$. 

Similarly, the image of $\iota_{k,k+1}$ is contained in $\mathcal{M}(n+k(k+1)/2,k)$ and   $X_{k,k+1}\subset \mathcal{M}(n+k(k+1)/2,k)$. Since both $X_{k,k+1}$ and $\mathcal{M}(n+k(k+1)/2,k)$ are irreducible varieties of dimension $2n$, we get $X_{k,k+1}=\mathcal{M}(n+k(k+1)/2,k)$.
\end{proof}



We can combine the diagrams \eqref{eq: def p q maps} into a longer chain:

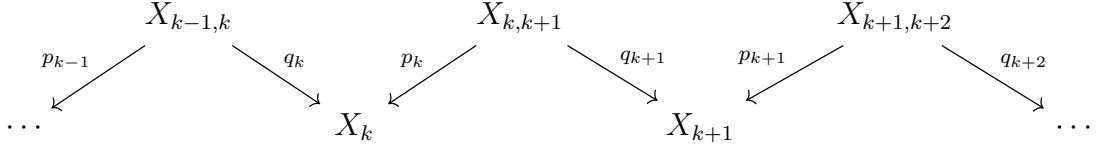
\begin{equation}\label{eq:mapdiagram}
    \begin{tikzcd}
    &X_{k-1,k} \arrow[dl,"p_{k-1}"'] \arrow[dr,"q_k"] && X_{k,k+1} \arrow[dl,"p_k"'] \arrow[dr,"q_{k+1}"] && X_{k+1,k+2} \arrow[dl,"p_{k+1}"']\arrow[dr,"q_{k+2}"] \\
    \cdots &&X_k && X_{k+1} && \cdots
\end{tikzcd}
\end{equation} 

\begin{example}
    For small $k$, these spaces have simple explicit descriptions. Indeed, we have $X_0 = \Hilb^n(\C^2)$, $X_{0,1} = \Hilb^{n,n+1}_0(\C^2)$ is the locus of pairs of ideals $(I,J)$ such that  $I/J$ is supported at the origin, and $X_1 \subset\Hilb^{n+1}(\C^2)$   is the locus of ideals $I\in \Hilb^{n+1}(\C^2)$ contained in $(x,y)$. Thus, the leftmost side of \eqref{eq:mapdiagram} is as follows:
    
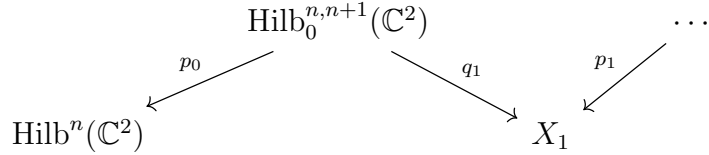
\[
        \begin{tikzcd}
        & \Hilb^{n,n+1}_0(\C^2) \arrow[dl,"p_0"'] \arrow[dr,"q_{1}"] && \cdots \arrow[dl,"p_{1}"'] \\
        \Hilb^n(\C^2) && X_1 
    \end{tikzcd}
    \]
\end{example}

\begin{remark}
The inductive argument in the proof of \cite[Theorem 3.3]{Bayer} shows that
all nonempty Brill-Noether loci $\BN^r(n)$ and their resolutions $\mathcal{M}(n,r)$ are birational to $\Hilb^{n'}(\C^2)$ for  $n'=n-r(r+1)/2$.
Proposition \ref{prop: X as BN} strengthens this observation and allows us to identify them with the explicit birational models of $\Hilb^{n'}(\C^2)$:
$$
\BN^r(n)=X_{r}(n'),\ \mathcal{M}(n,r)=X_{r,r+1}(n').
$$

\end{remark}

\section{Algebraic properties of the section ring}\label{sec:finitegenerationfacts}

In this section we establish several algebraic properties of the section ring 
\[ R(n)=\bigoplus_{k,m\geq 0} H^0\left(\Hilb^n\left(\Bl_{\mathbf{0}}\left(\C^2\right)\right),\CO(k,m)\right) \simeq \bigoplus_{k,m\geq 0} A^{m}_{\geq k}(n), \]
 including Theorem \ref{thm:finitegenerationfacts}.

\subsection{Division with remainder}\label{sec:division}

We will need several combinatorial facts about the sets $\mathcal{P}(k,m)$ from which we will derive algebraic properties about $R(n)$.
Throughout this section we will abbreviate
$$
(\aaa,\bbb)=\big( (a_1,b_1),\dots,(a_n,b_n) \big)\in \N^{2n}.
$$
We assume these points are ordered lexicographically, as with the $n$-element subset $S$ in Definition \ref{def: delta S}. Our main goal is the following proposition.

\begin{proposition}\label{prop:decomp}
    Let $k\geq 0$ and $m>0$, and write $k= qm+r$ with $0\leq r <m$. For any $(\aaa,\bbb) = \big( (a_1,b_1),\dots,(a_n,b_n) \big)\in \mathcal{P}(k,m)$, there is a decomposition $$(\aaa,\bbb) = \left(\aaa^{(0)},\bbb^{(0)}\right)+\cdots+\left(\aaa^{(m-1)},\bbb^{(m-1)}\right)$$ where $\left(\aaa^{(s)},\bbb^{(s)}\right)\in \mathcal{P}(q+1,1)$ for $s=0,\dots,r-1$ and $\left(\aaa^{(s)},\bbb^{(s)}\right)\in \mathcal{P}(q,1)$ for $s=r,\dots,m-1$.
\end{proposition}

Before the proof, we first state and derive the main consequence of this proposition we will use in later sections.

\begin{corollary}
    \label{cor:Minkowski decomp}
    For any integer $q\geq 0$ and integers $m_1,m_2\geq 0$ not both $0$, we have
    \[ \mathcal{P}(qm_1+(q+1)m_2,m_1+m_2) = \underbrace{\mathcal{P}(q,1)+\cdots+\mathcal{P}(q,1)}_{m_1 \text{ terms}}+\underbrace{\mathcal{P}(q+1,1)+\cdots \mathcal{P}(q+1,1)}_{m_2 \text{ terms}} \]
\end{corollary}

\begin{proof}
    One easily shows that for any $k,k',m,m'$ there is an inclusion
    \[ \mathcal{P}(k+k',m+m') \supseteq \mathcal{P}(k,m)+\mathcal{P}(k',m'). \]
    Repeated applications of this give the first desired inclusion. The other inclusion is precisely Proposition \ref{prop:decomp}. Indeed, for $m_1>0$ we can apply Proposition \ref{prop:decomp} with $k = qm_1+(q+1)m_2$, $m = m_1+m_2$, and $r=m_2$, while for $m_1=0$ we choose $k=(q+1)m_2$, $m=m_2$ and $r=0$.
\end{proof}

Now we will work towards the proof of Proposition \ref{prop:decomp}. First, we give a refined version of the decomposition of $\mathcal{P}(0,m)$, i.e. the case $k=0$. The decomposition is easiest to describe for points $(\aaa,\bbb)$ for which all the coordinates $b_i$ are minimal, i.e. when 
\[ b_i = \sum_{j=1}^{i-1}\max\{m-(a_i-a_j),0\} \]
for all $i$. The following lemma gives the desired decomposition depending only on the $a$-coordinates for such points. Later we will modify the $b$-coordinates   to give the decomposition for an arbitrary point $(\aaa,\bbb)\in \mathcal{P}(0,m).$

\begin{lemma}\label{lem:div with remainder}
    Fix integers $0\leq a_1\leq \cdots \leq a_n$ and $m> 0$, and write $a_i = mq_i+r_i$ where $0\leq r_i<m$. The quantities
    \[ a_i^{(s)}=\begin{cases}
    q_i+1 & s<r_i,\\
    q_i & s\ge r_i.
    \end{cases},
    \hspace{1cm} \text{and} \hspace{1cm}
    \tilde{b}^{(s)}_i = \big|\left\{ j=1,\dots,i-1 \, : \, a^{(s)}_j = a^{(s)}_i\right\}\big|, \]
    defined for all $i=1,\dots,n$ and $s=0,\dots,m-1$, satisfy $a^{(0)}_i+\cdots+a^{(m-1)}_i = a_i$ and 
    \[ \tilde{b}^{(0)}_i+\cdots+\tilde{b}^{(m-1)}_i = \sum_{j=1}^{i-1}\max\{m-(a_i-a_j),0\} \]
    for all $i=1,\dots,n$.
\end{lemma}

\begin{proof}
    This is essentially \cite[Lemma 6.20]{GKO}, but we repeat the proof here for the reader's convenience. For each $i=1,\dots,n$ we have
    $$
    a^{(0)}_i+\ldots+a^{(m-1)}_i=r_i(q_i+1)+(m-r_i)q_i=mq_i+r_i=a_i,
    $$
    so $a^{(0)}_i+\ldots+a^{(m-1)}_i = a_i$. For the second claim, observe that the sum $\tilde{b}^{(0)}_i+\cdots+\tilde{b}^{(m-1)}_i$ is equal to the number of pairs $(j,s)\in \{1,\dots,i-1\}\times\{0,\dots,m-1\}$ such that $a_j^{(s)}=a_i^{(s)}$. It therefore suffices then to show that for each  $j=1,\dots,i-1$, 
    \begin{equation}\label{eqn:number of equalities}
        \big|\left\{s = 0,\dots,m-1 \, : \, a^{(s)}_j = a^{(s)}_i \right\} \big| = \max\{m-(a_i-a_j),0\}. 
    \end{equation}
    We have the following cases:
    
    1) $q_i=q_j,r_i\ge r_j$. In this case $a_i^{(s)}=a_{j}^{(s)}=q_i+1$ for $s=0,\ldots,r_j-1$ and $a_i^{(s)}=a_{j}^{(s)}=q_i$  for $s=r_i,\ldots,m-1$, so the left hand side of \eqref{eqn:number of equalities} equals
    $$
    r_j+(m-r_i)=m-(r_i-r_j)=m-(a_i-a_j).
    $$

    2) $q_i=q_j+1,r_i<r_j$. In this case 
    $a_i^{(s)}=a_{j}^{(s)}=q_i=q_j+1$ for $s=r_i,\ldots,r_j-1$ and the left hand side of \eqref{eqn:number of equalities} equals
    $$
    r_j-r_i=m-(m(q_j+1)+r_i)-(mq_j+r_j))=m-(a_i-a_j).
    $$
    
    3) In all other cases we have $a_i-a_j\ge m$ and there is no overlap between $a_i^{(s)}$ and $a_j^{(s)}$. Indeed, for $q_i=q_j+1$ and $r_i\ge r_j$ we get $a_i^{(s)}=q_j+2$ for $s<r_j$ and $a_i^{(s)}\ge q_j+1$ for $s\ge r_j$. 
    For $q_i\ge q_j+2$ we have $a_i^{(s)}\ge q_j+2$ while $a_j^{(s)}\le q_j+1$. This completes the proof of the lemma.
\end{proof}

Now we give the desired decomposition for an arbitrary point in $\mathcal{P}(0,m)$ by modifying the $\widetilde{b}_i^{(s)}$-coordinates.

\begin{lemma}\label{lem:k=0 decomp}
    For any $(\aaa,\bbb) = \big( (a_1,b_1),\dots,(a_n,b_n) \big)\in \mathcal{P}(0,m)$, there is a decomposition $(\aaa,\bbb) = \left(\aaa^{(0)},\bbb^{(0)}\right)+\cdots+\left(\aaa^{(m-1)},\bbb^{(m-1)}\right)$ where $\left(\aaa^{(s)},\bbb^{(s)}\right)\in \mathcal{P}(0,1)$ and $a_i^{(s)}$ are as in Lemma \ref{lem:div with remainder}.
\end{lemma}

\begin{proof}
    Let $\left(\aaa^{(s)},\tilde{\bbb}^{(s)}\right)$ be as in Lemma \ref{lem:div with remainder}, and define
    \[ b^{(s)}_i = \begin{cases}
        \tilde{b}^{(s)}_i & s\neq r_i,\\
        \tilde{b}^{(s)}_i + b_i- \sum_{j=1}^{i-1}\max\{m-(a_i-a_j),0\} & s=r_i.
    \end{cases} \]
    By Lemma \ref{lem:div with remainder} we immediately have $a^{(0)}_i+\cdots+a^{(m-1)}_i = a_i$ and 
    \[ b^{(0)}_i+\cdots+b^{(m-1)}_i = \tilde{b}^{(0)}_i+\cdots+\tilde{b}^{(m-1)}_i + b_i- \sum_{j=1}^{i-1}\max\{m-(a_i-a_j),0\} =b_i \]
    for all $i$.  
    
    It remains to check that $\left(\aaa^{(s)},\bbb^{(s)}\right)\in \mathcal{P}(0,1)$, i.e. that the points $\left(a^{(s)}_1,b^{(s)}_1\right),\dots,\left(a^{(s)}_n,b^{(s)}_n\right)$ lie in $\Z^2_{\geq 0}$ and strictly increase in lexicographic order. The nonnegativity and integrality of $a^{(s)}_i$ and $\tilde{b}^{(s)}_i$ is clear, and the conditions defining $(\aaa,\bbb)\in \mathcal{P}(0,m)$ ensure that we also have $b^{(r_i)}_i -\tilde{b}^{(r_i)}_i \in \Z_{\geq 0}.$

    Finally, we check that $\left(a^{(s)}_i,b^{(s)}_i\right)<\left(a^{(s)}_{i+1},b^{(s)}_{i+1}\right)$ in lexicographic order. One sees from the construction that $a_i\leq a_{i+1}$ implies $a_i^{(s)}\leq a_{i+1}^{(s)}$ for all $s$. Now suppose that $s$ is such that $a_i^{(s)}= a_{i+1}^{(s)}$. In this case, we have 
    \[ \tilde{b}^{(s)}_{i+1} = \big|\left\{ j=1,\dots,i \, : \, a^{(s)}_j = a^{(s)}_{i+1} \right\}\big| = \big|\left\{ j=1,\dots,i-1 \, : \, a^{(s)}_j = a^{(s)}_{i} \right\}\cup \{i\} \big| = \tilde{b}^{(s)}_{i}+1. \]
    
    If we further have $s\neq r_i$, then $b^{(s)}_{i+1}\geq \tilde{b}^{(s)}_{i+1}> \tilde{b}^{(s)}_{i} = b^{(s)}_{i}$ as desired. 

    Now suppose on the other hand that $s=r_i$ so that $a_i^{(r_i)}= a_{i+1}^{(r_i)}$. We first claim that this combination of assumptions requires $a_i=a_{i+1}$. Indeed, if $a_i<a_{i+1}$ then we either have $q_i<q_{i+1}$ or $q_i=q_{i+1}$ and $r_i<r_{i+1}$. If $q_i<q_{i+1}$ then $a_i^{(r_i)} = q_i<q_{i+1}\leq a_{i+1}^{(r_i)}$, while if $q_i=q_{i+1}$ and $r_i<r_{i+1}$ then we have $a^{(r_i)}_{i+1} = q_{i+1}+1=q_i+1$, but $a^{(r_i)}_i = q_i$. This completes the proof of the claim that $a_i=a_{i+1}$ in this case.

    It follows that $s = r_i=r_{i+1}$, so that in this case we have
    \[ b^{(s)}_i = \tilde{b}^{(s)}_i + b_i- \sum_{j=1}^{i-1}\max\{m-(a_i-a_j),0\}, \]
    and 
    \[ b^{(s)}_{i+1} = \tilde{b}^{(s)}_{i+1} + b_{i+1}- \sum_{j=1}^{i}\max\{m-(a_{i+1}-a_j),0\}. \]
    When taking the difference $b^{(s)}_{i+1}-b^{(s)}_i$, the $j=1,\dots,i-1$ terms in the sums cancel leaving the $j=i$ term in the latter sum, which is $m$. We therefore compute
    \[ b^{(s)}_{i+1}-b^{(s)}_i = \tilde{b}^{(s)}_{i+1}-\tilde{b}^{(s)}_i+b_{i+1}-b_i-m. \]
    The same argument given above shows that $\tilde{b}^{(s)}_{i+1}-\tilde{b}^{(s)}_i = 1$, and we have $b_{i+1}-b_i\geq m$ by the definition of $\mathcal{P}(k,m)$. We therefore have $b^{(s)}_{i+1}-b^{(s)}_i \geq b_{i+1}-b_i-(m-1)\geq m-(m-1) = 1$. 
\end{proof}

\begin{example}\label{ex:decomp k=0}
    For this example we take $n=5$ and $m=3$. Consider the tuple $(a_1,\dots,a_5) = (0,2,2,3,4).$ Lemma \ref{lem:div with remainder} gives a decomposition of $(a_1,\dots,a_5)$ into the sequences $\left(a^{(0)}_1,\dots,a^{(0)}_5\right) = (0,1,1,1,2),$ $\left(a^{(1)}_1,\dots,a^{(1)}_5\right) = (0,1,1,1,1),$ and $\left(a^{(2)}_1,\dots,a^{(2)}_5\right) = (0,0,0,1,1)$. The $\widetilde{b}^{(s)}_i$ coordinates are the minimal nonnegative integers so that the points $\left(a^{(s)}_1,\widetilde{b}^{(s)}_1\right),\dots,\left(a^{(s)}_5,\widetilde{b}^{(s)}_5\right)$ strictly increase in lexicographic order. Specifically, we have
    \[ \left(\aaa^{(0)},\widetilde{\bbb}^{(0)}\right) = \left(\left(a^{(0)}_1,\widetilde{b}^{(0)}_1\right),\dots,\left(a^{(0)}_5,\widetilde{b}^{(0)}_5\right)\right) = ((0,0),(1,0),(1,1),(1,2),(2,0))\in \mathcal{P}(0,1), \]
    \[ \left(\aaa^{(1)},\widetilde{\bbb}^{(1)}\right) = \left(\left(a^{(1)}_1,\widetilde{b}^{(1)}_1\right),\dots,\left(a^{(1)}_5,\widetilde{b}^{(1)}_5\right)\right) = ((0,0),(1,0),(1,1),(1,2),(1,3))\in \mathcal{P}(0,1), \]
    \[ \left(\aaa^{(2)},\widetilde{\bbb}^{(2)}\right)=\left(\left(a^{(2)}_1,\widetilde{b}^{(2)}_1\right),\dots,\left(a^{(2)}_5,\widetilde{b}^{(2)}_5\right)\right) = ((0,0),(0,1),(0,2),(1,0),(1,1))\in \mathcal{P}(0,1). \]
    Taking the coordinate-wise sum of these tuples gives the tuple
    \[ \left(\aaa,\widetilde{\bbb}\right) =\left(\left(a_1,\widetilde{b}_1\right),\dots,\left(a_5,\widetilde{b}_5\right)\right) = ((0,0),(2,1),(2,4),(3,4),(4,4))\in \mathcal{P}(0,3). \]
    One checks these $\widetilde{b}_i$ coordinates are minimal with respect to the $a$-coordinates to have $(\aaa,\bbb)\in \mathcal{P}(0,3)$. In other words, for all $i=1,\dots,5$ we have 
    \[ \widetilde{b}_i = \sum_{j=1}^{i-1}\max\{ 3-(a_i-a_j),0 \}. \]
    
    \begin{figure}[h]
    \centering
    \begin{subfigure}{.2\textwidth}
        \begin{tikzpicture}[scale = .7]
        \draw[<->, thick] (3,0)--(0,0)--(0,4);
        \draw[thick] (1,.2)--(1,-.2) node[below] {$1$};
        \draw[thick] (2,.2)--(2,-.2) node[below] {$2$};
        \draw[thick] (.2,1)--(-.2,1) node[left] {$1$};
        \draw[thick] (.2,2)--(-.2,2) node[left] {$2$};
        \draw[thick] (.2,3)--(-.2,3) node[left] {$3$};
        \filldraw[black] (0,0) circle (3pt);
        \filldraw[black] (1,0) circle (3pt);
        \filldraw[black] (1,1) circle (3pt);
        \filldraw[black] (1,2) circle (3pt);
        \filldraw[black] (2,0) circle (3pt);
    \end{tikzpicture}
    \caption{$\left(\aaa^{(0)},\widetilde{\bbb}^{(0)}\right)$}
    \end{subfigure}%
    \begin{subfigure}{.2\textwidth}
        \begin{tikzpicture}[scale = .7]
        \draw[<->, thick] (3,0)--(0,0)--(0,4);
        \draw[thick] (1,.2)--(1,-.2) node[below] {$1$};
        \draw[thick] (2,.2)--(2,-.2) node[below] {$2$};
        \draw[thick] (.2,1)--(-.2,1) node[left] {$1$};
        \draw[thick] (.2,2)--(-.2,2) node[left] {$2$};
        \draw[thick] (.2,3)--(-.2,3) node[left] {$3$};
        \filldraw[black] (0,0) circle (3pt);
        \filldraw[black] (1,0) circle (3pt);
        \filldraw[black] (1,1) circle (3pt);
        \filldraw[black] (1,2) circle (3pt);
        \filldraw[black] (1,3) circle (3pt);
    \end{tikzpicture}
    \caption{$\left(\aaa^{(1)},\widetilde{\bbb}^{(1)}\right)$}
    \end{subfigure}%
    \begin{subfigure}{.2\textwidth}
        \begin{tikzpicture}[scale = .7]
        \draw[<->, thick] (3,0)--(0,0)--(0,4);
        \draw[thick] (1,.2)--(1,-.2) node[below] {$1$};
        \draw[thick] (2,.2)--(2,-.2) node[below] {$2$};
        \draw[thick] (.2,1)--(-.2,1) node[left] {$1$};
        \draw[thick] (.2,2)--(-.2,2) node[left] {$2$};
        \draw[thick] (.2,3)--(-.2,3) node[left] {$3$};
        \filldraw[black] (0,0) circle (3pt);
        \filldraw[black] (0,1) circle (3pt);
        \filldraw[black] (0,2) circle (3pt);
        \filldraw[black] (1,0) circle (3pt);
        \filldraw[black] (1,1) circle (3pt);
    \end{tikzpicture}
    \caption{$\left(\aaa^{(2)},\widetilde{\bbb}^{(2)}\right)$}
    \end{subfigure}%
    \begin{subfigure}{.3\textwidth}
        \begin{tikzpicture}[scale = .7]
        \draw[<->, thick] (5,0)--(0,0)--(0,5);
        \draw[thick] (1,.2)--(1,-.2) node[below] {$1$};
        \draw[thick] (2,.2)--(2,-.2) node[below] {$2$};
        \draw[thick] (3,.2)--(3,-.2) node[below] {$3$};
        \draw[thick] (4,.2)--(4,-.2) node[below] {$4$};
        \draw[thick] (.2,1)--(-.2,1) node[left] {$1$};
        \draw[thick] (.2,2)--(-.2,2) node[left] {$2$};
        \draw[thick] (.2,3)--(-.2,3) node[left] {$3$};
        \draw[thick] (.2,4)--(-.2,4) node[left] {$4$};
        \filldraw[black] (0,0) circle (3pt);
        \filldraw[black] (2,1) circle (3pt);
        \filldraw[black] (2,4) circle (3pt);
        \filldraw[black] (3,4) circle (3pt);
        \filldraw[black] (4,4) circle (3pt);
    \end{tikzpicture}
    \caption{$\left(\aaa,\widetilde{\bbb}\right)\in \mathcal{P}(0,3)$}
    \end{subfigure}%
    \caption{The decomposition of a point $(\aaa,\widetilde{\bbb})\in \mathcal{P}(0,3)$ with minimal $\widetilde{\bbb}$-coordinates.}
    \label{fig:P(0,k) decomp}
    \end{figure}
    For an arbitrary point in $\mathcal{P}(0,3)$ with this sequence of $a$-coordinates, the decomposition is constructed by increasing certain $\widetilde{b}^{(s)}_i$'s in such a way that preserves the lexicographic order of each tuple. For example, the decomposition of
    \[ ((a_1,b_1),\dots,(a_5,b_5)) = ((0,5),(2,1),(2,10),(3,7),(4,4)) \in \mathcal{P}(0,3). \]
    constructed following the proof of Lemma \ref{lem:k=0 decomp} is
    \[ \left(\left(a^{(0)}_1,b^{(0)}_1\right),\dots,\left(a^{(0)}_5,b^{(0)}_5\right)\right) = ((0,5),(1,0),(1,1),(1,5),(2,0))\in \mathcal{P}(0,1), \]
    \[ \left(\left(a^{(1)}_1,b^{(1)}_1\right),\dots,\left(a^{(1)}_5,b^{(1)}_5\right)\right) = ((0,0),(1,0),(1,1),(1,2),(1,3))\in \mathcal{P}(0,1), \]
    \[ \left(\left(a^{(2)}_1,b^{(2)}_1\right),\dots,\left(a^{(2)}_5,b^{(2)}_5\right)\right) = ((0,0),(0,1),(0,8),(1,0),(1,1))\in \mathcal{P}(0,1). \]
    This modification cannot be done arbitrarily. For example, we had to increase $\widetilde{b}^{(0)}_4$ since increasing either $\widetilde{b}^{(1)}_4$ or $\widetilde{b}^{(2)}_4$ would not maintain the lexicographic ordering. This is clear from Figure \ref{fig:P(0,k) decomp}, where one sees that only the collection $(\aaa^{(0)},\bbb^{(0)})$ has its fourth point as the highest point in its column.  
\end{example}

Now we can give the complete proof of Proposition \ref{prop:decomp}.

\begin{proof}[Proof of Proposition \ref{prop:decomp}]
    Given $(\aaa,\bbb)\in \mathcal{P}(k,m)$, we define $(\aaa,\overline{\bbb})= \big( (a_1,\overline{b}_1),\dots,(a_n,\overline{b}_n) \big)$ by $\overline{b}_i = b_i - \max\{ k-a_i,0 \}$. Comparing the definitions of $\mathcal{P}(k,m)$ and $\mathcal{P}(0,m)$ it is clear that this operation gives $(\aaa,\overline{\bbb})\in \mathcal{P}(0,m)$.

    Let $(\aaa,\overline{\bbb}) = \left(\aaa^{(0)},\overline{\bbb}^{(0)}\right)+\cdots+\left(\aaa^{(m-1)},\overline{\bbb}^{(m-1)}\right)$ be a decomposition as in Lemma \ref{lem:k=0 decomp}. For ease of notation, we set 
    \[ k^{(s)} = \begin{cases}
        q+1 & s< r,\\
        q & s\geq r.
    \end{cases} \]
    We now modify this decomposition of $(\aaa,\overline{\bbb})$ by setting $b^{(s)}_i = \overline{b}^{(s)}_i +\max\left\{k^{(s)} - a^{(s)}_i,0\right\}.$
    Again comparing the definitions of $\mathcal{P}(0,1)$ and $\mathcal{P}(k^{(s)},1)$, this operation clearly gives $\left(\aaa^{(s)},\bbb^{(s)}\right)\in \mathcal{P}\left(k^{(s)},1\right)$ for all $s$ as desired.
    
    It remains to check that we have $\left(\aaa,\bbb\right) = \left(\aaa^{(0)},\bbb^{(0)}\right)+\cdots+\left(\aaa^{(m-1)},\bbb^{(m-1)}\right)$. The equality on $\aaa$'s is given by the choice of the decomposition of $\left(\aaa,\overline{\bbb}\right)$ as in Lemma \ref{lem:k=0 decomp}. Toward the equality for $\bbb$'s, we similarly have by construction 
    \[ \overline{b}^{(0)}_i+\cdots+\overline{b}^{(m-1)}_i = \overline{b}_i = b_i - \max\{ k-a_i,0 \}. \]
    By the choice of $b_i^{(s)}$ we also have 
    \[ b_i^{(0)} + \cdots b_i^{(m-1)} = \overline{b}^{(0)}_i+\cdots+\overline{b}^{(m-1)}_i + \sum_{s=0}^{m-1} \max\left\{k^{(s)} - a^{(s)}_i,0\right\}. \]
    It therefore suffices to show 
    \begin{equation}\label{eq:differences div with rem}
        \max\{ k-a_i,0 \} = \sum_{s=0}^{m-1} \max\left\{k^{(s)} - a^{(s)}_i,0\right\}. 
    \end{equation} 
    Indeed, if $k\geq a_i$ then we have $k^{(s)}\geq a^{(s)}_i$ for all $s$. In this case, we can compute
    \[ \sum_{s=0}^{m-1} \max\left\{k^{(s)} - a^{(s)}_i,0\right\} = \sum_{s=0}^{m-1}\left(k^{(s)}-a_i^{(s)}\right) = k-a_i = \max\{ k-a_i,0 \} \]
    as desired. In the other case we have $k<a_i$ and $k^{(s)}\leq a^{(s)}_i$ for all $s$ and so both sides of \eqref{eq:differences div with rem} are zero. This completes the proof of the proposition.
\end{proof}

\begin{example}
    In this example we again take $n=5$ and $m=3$. We describe the decomposition of the tuple
    \[ ((a_1,b_1),\dots,(a_5,b_5)) = ((0,10),(2,4),(2,13),(3,9),(4,5)) \in \mathcal{P}(5,3) \]
    constructed in the proof of Proposition \ref{prop:decomp}. First, the transformed tuple is
    \[ ((a_1,\overline{b}_1),\dots,(a_5,\overline{b}_5)) = ((0,5),(2,1),(2,10),(3,7),(4,4)) \in \mathcal{P}(0,3). \]
    The decomposition of this tuple was given in Example \ref{ex:decomp k=0}. In this example we have $k=5$ so that $k^{(0)} = 2$, $k^{(1)} = 2$, and $k^{(2)} = 1$. We then set $b^{(s)}_i = \overline{b}^{(s)}_i +\max\left\{k^{(s)} - a^{(s)}_i,0\right\},$ which gives
    \[ \left(\left(a^{(0)}_1,b^{(0)}_1\right),\dots,\left(a^{(0)}_5,b^{(0)}_5\right)\right) = ((0,7),(1,1),(1,2),(1,6),(2,0)) \in \mathcal{P}(2,1), \]
    \[ \left(\left(a^{(1)}_1,b^{(1)}_1\right),\dots,\left(a^{(1)}_5,b^{(1)}_5\right)\right) = ((0,2),(1,1),(1,2),(1,3),(1,4))\in \mathcal{P}(2,1), \]
    \[ \left(\left(a^{(2)}_1,b^{(2)}_1\right),\dots,\left(a^{(2)}_5,b^{(2)}_5\right)\right) = ((0,1),(0,2),(0,9),(1,0),(1,1))\in \mathcal{P}(1,1). \]
    One checks that the coordinate-wise sum of these points gives the original tuple in $\mathcal{P}(5,3)$.

    \begin{figure}[h]
    \centering
    \begin{subfigure}{.2\textwidth}
        \begin{tikzpicture}[scale = .7]
        \draw[<->, thick] (3,0)--(0,0)--(0,4);
        \draw[thick] (1,.2)--(1,-.2) node[below] {$1$};
        \draw[thick] (2,.2)--(2,-.2) node[below] {$2$};
        \draw[thick] (.2,1)--(-.2,1) node[left] {$1$};
        \draw[thick] (.2,2)--(-.2,2) node[left] {$2$};
        \draw[thick] (.2,3)--(-.2,3) node[left] {$3$};
        \filldraw[black] (0,2) circle (3pt);
        \filldraw[black] (1,1) circle (3pt);
        \filldraw[black] (1,2) circle (3pt);
        \filldraw[black] (1,3) circle (3pt);
        \filldraw[black] (2,0) circle (3pt);
    \end{tikzpicture}
    \caption{$\left(\aaa^{(0)},\bbb^{(0)}\right)$}
    \end{subfigure}%
    \begin{subfigure}{.2\textwidth}
        \begin{tikzpicture}[scale = .7]
        \draw[<->, thick] (3,0)--(0,0)--(0,4);
        \draw[thick] (1,.2)--(1,-.2) node[below] {$1$};
        \draw[thick] (2,.2)--(2,-.2) node[below] {$2$};
        \draw[thick] (.2,1)--(-.2,1) node[left] {$1$};
        \draw[thick] (.2,2)--(-.2,2) node[left] {$2$};
        \draw[thick] (.2,3)--(-.2,3) node[left] {$3$};
        \filldraw[black] (0,2) circle (3pt);
        \filldraw[black] (1,1) circle (3pt);
        \filldraw[black] (1,2) circle (3pt);
        \filldraw[black] (1,3) circle (3pt);
        \filldraw[black] (1,4) circle (3pt);
    \end{tikzpicture}
    \caption{$\left(\aaa^{(1)},\bbb^{(1)}\right)$}
    \end{subfigure}%
    \begin{subfigure}{.2\textwidth}
        \begin{tikzpicture}[scale = .7]
        \draw[<->, thick] (3,0)--(0,0)--(0,4);
        \draw[thick] (1,.2)--(1,-.2) node[below] {$1$};
        \draw[thick] (2,.2)--(2,-.2) node[below] {$2$};
        \draw[thick] (.2,1)--(-.2,1) node[left] {$1$};
        \draw[thick] (.2,2)--(-.2,2) node[left] {$2$};
        \draw[thick] (.2,3)--(-.2,3) node[left] {$3$};
        \filldraw[black] (0,1) circle (3pt);
        \filldraw[black] (0,2) circle (3pt);
        \filldraw[black] (0,3) circle (3pt);
        \filldraw[black] (1,0) circle (3pt);
        \filldraw[black] (1,1) circle (3pt);
    \end{tikzpicture}
    \caption{$\left(\aaa^{(2)},\bbb^{(2)}\right)$}
    \end{subfigure}%
    \label{fig:P(0,k) decomp not minimal}
    \caption{The decomposition of a point $(\aaa,\bbb)\in \mathcal{P}(5,3)$ }
    \end{figure}
\end{example}

We will need one additional decomposition lemma.

\begin{lemma}\label{lem: m=1 decomp}
    For all $k\geq n-1$ we have $\mathcal{P}(k,1)+\mathcal{P}(1,0)= \mathcal{P}(k+1,1)$.
\end{lemma}

\begin{proof}
    The inclusion $ \mathcal{P}(k,1)+\mathcal{P}(1,0)\subseteq \mathcal{P}(k+1,1)$ is clear from the additivity of the conditions defining $\mathcal{P}(k,m)$. Toward the other inclusion, take $(\aaa,\bbb) = \big((a_1,b_1),\dots,(a_n,b_n)\big)\in \mathcal{P}(k+1,1)$. Since $k+2>n$ by assumption, there is some $a = 0,\dots,k+1$ such that none of the $a_i$'s are equal to $a$, and we set $N = |\left\{i=1,\dots,n \, :\ \, a_i<a\right\}|$. Then for $N<n$ we get
    $$
    0\le a_1\le \cdots\le a_N<a<a_{N+1}\le \cdots \le a_n.
    $$
    and for $N=n$ we get
    $$
    0\le a_1\le \cdots\le a_n<a.
    $$
    For $N<n$  we may decompose $(\aaa,\bbb)$ as the sum of the tuple
    \[ (\aaa',\bbb') = \big( (a_1,b_1-1),\dots (a_N,b_N-1),(a_{N+1}-1,b_{N+1}),\dots,(a_n-1,b_n)\big) \in \mathcal{P}(k,1) \]
    and the tuple
    \[ (\aaa'',\bbb'') = \big( (0,1),\dots (0,1),(1,0),\dots,(1,0)\big) \in \mathcal{P}(1,0), \]
    where in the latter tuple there are $N$ points equal to $(0,1)$. Note that  for $1\le i\le N$ we have $a_i<a\le k+1$ and $a_i+b_i\ge k+1$, so $b_i>0$. Also,
    $a_{N}<a<a_{N+1}$ implies $a_{N+1}-1>a_{N}$, so the points in $(\aaa',\bbb')$ are strictly lexicographically ordered.  

    For $N=n$ we decompose  decompose $(\aaa,\bbb)$ as the sum of the tuple
    \[ (\aaa',\bbb') = \big( (a_1,b_1-1),\dots (a_n,b_n-1) \big) \in \mathcal{P}(k,1) \]
    and the tuple
    \[ (\aaa'',\bbb'') = \big( (0,1),\dots (0,1)\big) \in \mathcal{P}(1,0). \]
\end{proof}

\begin{example}
    We illustrate the decomposition in Lemma \ref{lem: m=1 decomp} in the case $n=3$. Consider the triple of points $(\aaa,\bbb) = \big( (0,3),(1,2),(3,0) \big)\in \mathcal{P}(1,3)$. The only $a=0,1,2,3$ for which $a_i\neq a$ for all $i$ is $a=2$. 
    \begin{figure}[h]
    \centering
    \begin{subfigure}{.35\textwidth}
        \begin{tikzpicture}[scale = .7]
        \draw[<->, thick] (5,0)--(0,0)--(0,5);
        \draw[thick] (1,.2)--(1,-.2) node[below] {$1$};
        \draw[thick] (2,.2)--(2,-.2) node[below] {$2$};
        \draw[thick] (3,.2)--(3,-.2) node[below] {$3$};
        \draw[thick] (4,.2)--(4,-.2) node[below] {$4$};
        \draw[thick] (.2,1)--(-.2,1) node[left] {$1$};
        \draw[thick] (.2,2)--(-.2,2) node[left] {$2$};
        \draw[thick] (.2,3)--(-.2,3) node[left] {$3$};
        \draw[thick] (.2,4)--(-.2,4) node[left] {$4$};
        \filldraw[opacity = .4] (0,0)--(0,3)--(3,0)--cycle;
        \filldraw[black] (0,3) circle (3pt) node[above right]{$p_1$};
        \filldraw[black] (1,2) circle (3pt) node[above right]{$p_2$};
        \filldraw[black] (3,0) circle (3pt) node[above right]{$p_3$};
    \end{tikzpicture}
    \caption{$(a,b)\in\mathcal{P}(1,3)$}
    \end{subfigure}%
    \hspace{.03\textwidth}
    \begin{subfigure}{.25\textwidth}
        \begin{tikzpicture}[scale = .7]
        \draw[<->, thick] (5,0)--(0,0)--(0,5);
        \draw[thick] (1,.2)--(1,-.2) node[below] {$1$};
        \draw[thick] (2,.2)--(2,-.2) node[below] {$2$};
        \draw[thick] (3,.2)--(3,-.2) node[below] {$3$};
        \draw[thick] (4,.2)--(4,-.2) node[below] {$4$};
        \draw[thick] (.2,1)--(-.2,1) node[left] {$1$};
        \draw[thick] (.2,2)--(-.2,2) node[left] {$2$};
        \draw[thick] (.2,3)--(-.2,3) node[left] {$3$};
        \draw[thick] (.2,4)--(-.2,4) node[left] {$4$};
        \filldraw[opacity = .4] (0,0)--(0,2)--(2,0)--cycle;
        \filldraw[black] (0,2) circle (3pt) node[above right]{$p_1$};
        \filldraw[black] (1,1) circle (3pt) node[above right]{$p_2$};
        \filldraw[black] (2,0) circle (3pt) node[above right]{$p_3$};
    \end{tikzpicture}
        \caption{$(a',b')\in\mathcal{P}(1,2)$}
            \end{subfigure}%
    \hspace{.03\textwidth}
    \begin{subfigure}{.25\textwidth}
        \begin{tikzpicture}[scale = .7]
        \draw[<->, thick] (5,0)--(0,0)--(0,5);
        \draw[thick] (1,.2)--(1,-.2) node[below] {$1$};
        \draw[thick] (2,.2)--(2,-.2) node[below] {$2$};
        \draw[thick] (3,.2)--(3,-.2) node[below] {$3$};
        \draw[thick] (4,.2)--(4,-.2) node[below] {$4$};
        \draw[thick] (.2,1)--(-.2,1) node[left] {$1$};
        \draw[thick] (.2,2)--(-.2,2) node[left] {$2$};
        \draw[thick] (.2,3)--(-.2,3) node[left] {$3$};
        \draw[thick] (.2,4)--(-.2,4) node[left] {$4$};
        \filldraw[opacity = .4] (0,0)--(0,1)--(1,0)--cycle;
        \filldraw[black] (0,1) circle (3pt) node[above right]{$p_1=p_2$};
        \filldraw[black] (1,0) circle (3pt) node[above right]{$p_3$};
    \end{tikzpicture}
        \caption{$(a'',b'')\in \mathcal{P}(0,1)$}
    \end{subfigure}
    \end{figure}
\end{example}

The constraint $k\geq n-1$ is necessary for Lemma \ref{lem:k=0 decomp} as the following lemma shows.

\begin{lemma}\label{lem:strict inclusions}
    For all $k<n-1$, we have $A^1_{\geq k}\cdot A^0_{\geq 1} \subsetneq A^1_{\geq k+1}$.
\end{lemma}

\begin{proof}
    We make use of the grading on $A^m_{\geq k}$ by total degree, i.e. the grading on the sets $\mathcal{P}(m,k)$ by $a_1+\cdots+a_n+b_1+\cdots+b_n$. The smallest nonzero total degree of an element in $A^0_{\geq 1}$ is easily seen to be $n$, obtained for example by $x_1\cdots x_n$ corresponding to $\{(1,0),\dots,(1,0)\} \in \mathcal{P}(0,1)$. To describe the smallest nonzero total degree of an element of $A^1_{\geq k}$, we let $r\geq 0$ and $0\leq s \leq k+r$ be such that
    \[ n = (k+1)+(k+2)+\cdots+(k+r)+s. \]
    The smallest degree of an element of $A^1_{\geq k}$ is 
    \[d(k) = k(k+1)+(k+1)(k+2)+\cdots+(k+r-1)(k+r)+(k+r)s,\]
    obtained by $\Delta_S$ where $S\subseteq \N^2_{\geq k}$ contains all points $(a,b)$ with $a+b = k,k+1,\dots,k+r-1$ and any $s$ points $(a,b)$ with $a+b = k+r$. Note that if $k\geq n-1$ then we have $r=0$ and $n=s$ corresponding to a choice of $S\subseteq \N^2_{\geq k}$ where every point $(a,b)\in S$ has $a+b=k$, while if $k<n-1$ then $S$ necessarily contains a point $(a,b)$ with $a+b>k$.

    We claim that for $k<n-1$ we have $d(k+1)<d(k)+n$, so that $A^1_{\geq k+1}$ contains elements of smaller total degree than does $A^1_{\geq k}\cdot A^0_{\geq 1}$. Indeed, we can give an element of $A^1_{\geq k+1}$ of degree $d(k)+n$ by multiplying some determinant $\Delta_S\in A^1_{\geq k}$ of smallest degree by $x_1\cdots x_n$. If we write $S = \{ (a_1,b_1),\dots,(a_n,b_n)\}$, then the resulting element of $A^1_{\geq k+1}$ is $\Delta_{S'}$ where $S' = \{ (a_1+1,b_1),\dots,(a_n+1,b_n)\}$. But if $k<n-1$, then $S'$ will include some point $(a,b)$ with $a+b>k+1$ and we can give a determinant in $A^1_{\geq k+1}$ of smaller total degree than $\Delta_{S'}$ by replacing this point of $S'$ with $(0,k+1)$. This completes the proof.
\end{proof}

\begin{example}
    We illustrate the failure of equality between $A^1_{\geq 2}\cdot A^0_{\geq 1}$ and $A^1_{\geq 3}$ in the case $n=9$. The smallest total degree of an element of $A^1_{\geq 2}$ is $2\cdot 3+3\cdot 4+4\cdot 2 = 26$, corresponding to the decomposition $9 = 3+4+2$. One possible choice of subset $S\subseteq \N^2_{\geq 2}$ for which $\Delta_S$ has degree $26$ is given below. The minimum total degree of an element of $A^1_{\geq 2}\cdot A^0_{\geq 1x}$ is therefore $26+9 = 35$, and one can give a determinant $\Delta_{S'}$ of this degree by setting $S' = \{ (a+1,b) \, | \, (a,b)\in S\}$. However, there are elements of $A^1_{\geq 3}$ of smaller total degree than this, for example $\Delta_{S''}$ where $S''= S'\cup\{ (0,3) \}\setminus \{ (2,3) \}$. This exhibits the nonequality $A^1_{\geq 2}\cdot A^0_{\geq 1} \subsetneq A^1_{\geq 3}$.
    \begin{figure}[h]
    \centering
    \begin{subfigure}{.35\textwidth}
        \begin{tikzpicture}[scale = .7]
        \draw[<->, thick] (5,0)--(0,0)--(0,5);
        \draw[thick] (1,.2)--(1,-.2) node[below] {$1$};
        \draw[thick] (2,.2)--(2,-.2) node[below] {$2$};
        \draw[thick] (3,.2)--(3,-.2) node[below] {$3$};
        \draw[thick] (4,.2)--(4,-.2) node[below] {$4$};
        \draw[thick] (.2,1)--(-.2,1) node[left] {$1$};
        \draw[thick] (.2,2)--(-.2,2) node[left] {$2$};
        \draw[thick] (.2,3)--(-.2,3) node[left] {$3$};
        \draw[thick] (.2,4)--(-.2,4) node[left] {$4$};
        \filldraw[opacity = .4] (0,0)--(0,2)--(2,0)--cycle;
        \filldraw[black] (0,2) circle (3pt) node[above right]{};
        \filldraw[black] (1,1) circle (3pt) node[above right]{};
        \filldraw[black] (2,0) circle (3pt) node[above right]{};
        \filldraw[black] (0,3) circle (3pt) node[above right]{};
        \filldraw[black] (1,2) circle (3pt) node[above right]{};
        \filldraw[black] (2,1) circle (3pt) node[above right]{};
        \filldraw[black] (3,0) circle (3pt) node[above right]{};
        \filldraw[black] (1,3) circle (3pt) node[above right]{};
        \filldraw[black] (2,2) circle (3pt) node[above right]{};
    \end{tikzpicture}
    \caption{$S\in\mathcal{P}(1,2)$}
    \end{subfigure}%
    \hspace{.03\textwidth}
        \begin{subfigure}{.35\textwidth}
        \begin{tikzpicture}[scale = .7]
        \draw[<->, thick] (5,0)--(0,0)--(0,5);
        \draw[thick] (1,.2)--(1,-.2) node[below] {$1$};
        \draw[thick] (2,.2)--(2,-.2) node[below] {$2$};
        \draw[thick] (3,.2)--(3,-.2) node[below] {$3$};
        \draw[thick] (4,.2)--(4,-.2) node[below] {$4$};
        \draw[thick] (.2,1)--(-.2,1) node[left] {$1$};
        \draw[thick] (.2,2)--(-.2,2) node[left] {$2$};
        \draw[thick] (.2,3)--(-.2,3) node[left] {$3$};
        \draw[thick] (.2,4)--(-.2,4) node[left] {$4$};
        \filldraw[opacity = .4] (0,0)--(0,3)--(3,0)--cycle;
        \filldraw[black] (1,2) circle (3pt) node[above right]{};
        \filldraw[black] (2,1) circle (3pt) node[above right]{};
        \filldraw[black] (3,0) circle (3pt) node[above right]{};
        \filldraw[black] (1,3) circle (3pt) node[above right]{};
        \filldraw[black] (2,2) circle (3pt) node[above right]{};
        \filldraw[black] (3,1) circle (3pt) node[above right]{};
        \filldraw[black] (4,0) circle (3pt) node[above right]{};
        \filldraw[black] (2,3) circle (3pt) node[above right]{};
        \filldraw[black] (3,2) circle (3pt) node[above right]{};
        \draw[thick,dashed, ->] (2,3) to [out=150,in=30] (0,3);
    \end{tikzpicture}
    \caption{$S'\in\mathcal{P}(1,3)$}
    \end{subfigure}
    \end{figure}
\end{example}

\subsection{Finiteness properties of the section ring}

Here we record several algebraic consequences of our combinatorial analysis in the previous section. 

\begin{theorem}\label{thm:infinite chambers fg}
    For any integer $k\geq 0$, the bigraded section ring
    \[ \bigoplus_{d_1,d_2\geq 0} H^0\left(\Hilb^n\left(\Bl_{\mathbf{0}}\left(\C^2\right)\right),\CO\left(d_1 k+d_2(k+1),d_1+d_2\right)\right) \simeq \bigoplus_{d_1,d_2\geq 0} A^{d_1+d_2}_{\geq d_1 k+d_2(k+1)}(n) \]
    is generated in bidegrees $(d_1,d_2) = (1,0)$ and $(0,1)$.
\end{theorem}

\begin{proof}
    There are $\C$-linear bases of $A^1_{\geq k}$ and $A^1_{\geq k+1}$ corresponding to the determinants $\Delta_S$ for $S=\{(a_1,b_1),\dots,(a_n,b_n)\}$ in $\mathcal{P}(k,1)$ and $\mathcal{P}(k+1,1)$ respectively. By Corollary \ref{cor:Minkowski decomp}, we have the equality of Minkowski sums
    \[ \mathcal{P}(d_1 k+d_2(k+1),d_1+d_2) = \underbrace{\mathcal{P}(k,1)+\cdots+\mathcal{P}(k,1)}_{d_1 \text{ terms}}+\underbrace{\mathcal{P}(k+1,1)+\cdots \mathcal{P}(k+1,1)}_{d_2 \text{ terms}} \]
    Taking products of the corresponding basis elements of $A^1_{\geq k}$ and $A^1_{\geq k+1}$, this shows that every element of $\mathcal{P}(d_1 k+d_2(k+1),d_1+d_2)$ can be realized as the trailing term of some polynomial in $\left( A^1_{\geq k}\right)^{d_1} \cdot \left( A^1_{\geq k+1}\right)^{d_2} \subseteq A^{d_1+d_2}_{\geq d_1 k+d_2(k+1)}$. By Theorem \ref{thm:HS lead terms}, these are all the trailing terms of polynomials in $A^{d_1+d_2}_{\geq d_1 k+d_2(k+1)}$, which implies 
    \[ \left( A^1_{\geq k}\right)^{d_1} \cdot \left( A^1_{\geq k+1}\right)^{d_2} = A^{d_1+d_2}_{\geq d_1 k+d_2(k+1)} \]
    as desired.
\end{proof}

Theorem \ref{thm:infinite chambers fg} refers to the subring of the total section ring consisting of line bundles in the cone generated by $\CO(k,1)$ and $\CO(k+1,1)$. This gives a decomposition of $\R^2_{\geq 0}$ into an infinite collection of chambers as shown in Figure \ref{fig:infinite chambers fg}.

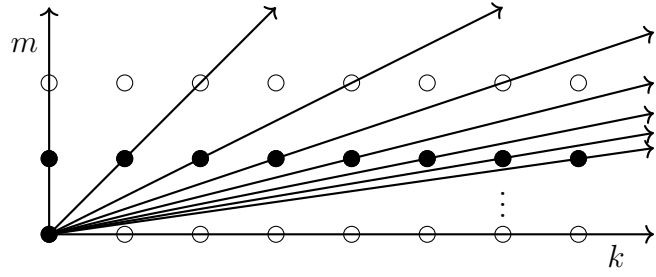
\begin{figure}[h]
    \centering
    \begin{tikzpicture}
        \draw[<-,thick] (8,0)--(0,0);
        \draw[->,thick] (0,0)--(0,3);
        \draw[->,thick] (0,0) -- (3,3);
        \draw[->,thick] (0,0) -- (6,3);    
        \draw[->,thick] (0,0) -- (8,2.67);
        \draw[->,thick] (0,0) -- (8,2); 
        \draw[->,thick] (0,0) -- (8,1.6);
        \draw[->,thick] (0,0) -- (8,1.34);
        \draw[->,thick] (0,0) -- (8,1.14); 
        \node at (6,.5) {$\vdots$};
        \node[left] at (0,2.5) {$m$};
        \node[below] at (7.5,0) {$k$};
        \foreach \x in {0,1,...,7}
            \foreach \y in {0,1,...,2}{
                \draw (\x,\y) circle (3pt);
            };
        \filldraw (0,0) circle (3pt);
        \filldraw (0,1) circle (3pt);
        \filldraw (1,1) circle (3pt);
        \filldraw (2,1) circle (3pt);
        \filldraw (3,1) circle (3pt);
        \filldraw (4,1) circle (3pt);
        \filldraw (5,1) circle (3pt);
        \filldraw (6,1) circle (3pt);
        \filldraw (7,1) circle (3pt);
    \end{tikzpicture}
    \caption{The chambers on which Theorem \ref{thm:infinite chambers fg} applies.}
    \label{fig:infinite chambers fg}
\end{figure}

As a consequence, we obtain the following finite-generation for singly-graded subrings.

\begin{corollary}\label{cor:fg on rays}
    For any integers $k,m\geq 0$, the singly-graded section ring 
    \[ \bigoplus_{d\geq 0} H^0\left(\Hilb^n\left(\Bl_{\mathbf{0}}\left(\C^2\right)\right),\CO(dk,dm)\right) \simeq \bigoplus_{d\geq 0} A^{dm}_{\geq dk}(n) \]
    is generated in degree $d = 1$.
\end{corollary}

 Similarly to Theorem \ref{thm:infinite chambers fg}, Lemma \ref{lem: m=1 decomp} implies the following statement.

\begin{theorem}\label{thm:finite chamber}
    The bigraded section ring
       \[ \bigoplus_{d_1,d_2\geq 0} H^0\left(\Hilb^n\left(\Bl_{\mathbf{0}}\left(\C^2\right)\right),\CO\left(d_1(n-1)+d_2,d_1\right)\right) \simeq \bigoplus_{d_1,d_2\geq 0} A^{d_1}_{\geq d_1(n-1)+d_2}(n) \]
    is generated in bidegrees $(d_1,d_2) = (1,0)$ and $(0,1)$.
\end{theorem}

Theorem \ref{thm:finite chamber} refers to the subring of the total section ring consisting of line bundles in the cone generated by $\CO(n-1,1)$ and $\CO(1,0)$, and subsumes Theorem \ref{thm:infinite chambers fg} in the range $k\geq n-1$. When combined, these results give a finite list of generating bidegrees for the total section ring.

\begin{corollary}\label{cor:finite generation degrees}
    The bigraded section ring 
    \[ \bigoplus_{k,m\geq 0} H^0(\Hilb^n(\Bl_{\mathbf{0}}(\C^2)),\CO(k,m)) \simeq \bigoplus_{k,m\geq 0} A^{m}_{\geq k}(n) \]
    is generated in bidegrees $(k,m)= (0,1),(1,1),\dots,(n-1,1)$ and $(1,0)$, and this list of bidegrees is minimal.
\end{corollary}


\begin{proof}
    That this ring is generated by its elements of the listed bidegrees follows from Theorems \ref{thm:infinite chambers fg} and \ref{thm:finite chamber}, and the minimality of this list follows from Lemma \ref{lem:strict inclusions}.
\end{proof}

\section{The birational models of $\Hilb^n(\Bl_{\mathbf{0}}(\C^2))$}\label{sec:proof of main thm}

\subsection{Proof of main theorems}

Recall that we have the map
\[ \iota_k: \Hilb^n(\C^2\setminus\mathbf{0})\to \Hilb^{n+k(k+1)/2}(\C^2) \]
defined by $I\mapsto (x,y)^k \cap I$, whose image closure is $X_k(n)$ by definition. We consider 
\[ \Hilb^n(\C^2\setminus\mathbf{0}) \subseteq \Hilb^n(\C^2) \]
as an open subset whose complement has codimension $2$. This implies that the restriction maps on Picard groups and global sections of line bundles are isomorphisms. We therefore also write $\CO(m)$ for the restriction of this line bundle from $\Hilb^n(\C^2)$ to $\Hilb^n(\C^2\setminus\mathbf{0})$, and identify the sections with $A^m(n)$.

\begin{proposition}\label{prop:image of pullback}
    Pullback by $\iota_k$ defines a ring homomorphism
    \begin{equation}\label{eq:ringmap}
         \iota_k^*: \bigoplus_{m\geq 0} A^m\left(n+\frac{k(k+1)}{2}\right) \to \bigoplus_{m\geq 0} A^m(n).
    \end{equation}
    whose image is
    \[ \bigoplus_{m\geq 0} A^m_{\geq mk}(n). \]
\end{proposition}

\begin{proof}
    One easily checks that the pullback of $\CO_{\Hilb^{n+k(k+1)/2}(\C^2)}(m)$ is $\CO_{\Hilb^{n}(\C^2\setminus\mathbf{0})}(m)$. 
    The codomain of this ring map is generated in degree $m=1$ by definition, and the claimed image is generated in degree $1$ by Corollary \ref{cor:fg on rays}. It therefore suffices to check the claim for $m=0,1$.

    For $m=0,1$, there are natural inclusions
    \begin{equation}\label{eq:decomp}
        A^m\left(n+\frac{k(k+1)}{2}\right) \subseteq A^m(n)\otimes A^m\left(\frac{k(k+1)}{2}\right), 
    \end{equation} 
    where we consider the first factor in the variables $x_i,y_i$ for $i=1,\dots,n$ and the second factor in the variables $x_i,y_i$ for $i=n+1,\dots,n+k(k+1)/2$. We can then compute the pullback by $\iota_k$ of any such section by evaluating the second factors at the scheme $Z_k$ corresponding to the monomial ideal $(x,y)^k$.

    In degree $m=0$ this is straightforward as the sections depend only on the support of the scheme, which in this case is entirely at the origin. In other words, the pullback map in degree $0$ is given by setting $x_i=y_i=0$ for all $i=n+1,\dots,n+k(k+1)/2$. It is clear that the image of this degree $0$ piece is all of $A^0(n)$, as claimed.

    Degree $m=1$ requires more care since the sections depend on the schemes themselves rather than just their supports. First, observe that for any subset $S\subseteq \Z^2_{\geq 0}$ with $|S| = n+k(k+1)/2$, the decomposition of $\Delta_S = \det (x_j^{a_i}y_j^{b_i})$ as in \eqref{eq:decomp} is given by the cofactor expansion along the first $n$ columns. This gives an expansion of the form
    \[ \Delta_S = \sum_{\substack{S'\sqcup S'' = S\\ |S'| = n, \, |S''| = k(k+1)/2}} \Delta_{S'} \otimes \Delta_{S''} \]
    where we recall that each determinant is defined only up to sign. 

    Next, we need to explain what we mean by evaluation at $Z_k$ for $m=1$. Consider the subset $S_k=\{ (a,b)\in \Z^2_{\geq 0} \,|\, a+b< k \}$, then $\Delta_{S_k}$ defines a section of $\CO(1)$ on $\Hilb^{k(k+1)/2}$ which does not vanish at $Z_K$.
.   We have an embedding 
    $$
    A\left(\frac{k(k+1)}{2}\right)\simeq \frac{1}{\Delta_{S_k}}A\left(\frac{k(k+1)}{2}\right)\subset \C\left(x_i,y_i:1\le i\le k(k+1)/2\right).
    $$
    The minor $\Delta_{S''}$ is sent to $\Delta_{S''}/\Delta_{S_k}$.
    Now we claim that 
    \[ \Delta_{S''}(Z_k)/\Delta_{S_k}(Z_K) = \begin{cases}
        1 & S'' = S_k
        \\
        0 & \text{else}.
    \end{cases} \]
    This is clear from Haiman's description of these sections in local coordinates \cite{Haimanqt}. It follows that the degree $m=1$ part of the image of $\iota_k^*$ is spanned by determinants $\Delta_{S'}$ where $S'\subseteq \Z^2_{\geq 0}$ has $a+b\geq k$ for all $(a,b)\in S'$. This image is therefore $A^1_{\geq k}(n),$ which generates the desired ring completing the proof.
\end{proof}

Now we may prove Theorem \ref{thm:main1} identifying $X_k$ as a certain birational model of $\Hilb^n(\Bl_{\mathbf{0}}(\C^2))$.

\begin{proof}[Proof of Theorem \ref{thm:main1}]
    Since the image of $\iota_k$ is dense in $X_k$, pullback by $\iota_k$ coincides with restriction to $X_k$. The kernel of $\iota_k^*$ is therefore the ideal $I(X_k(n))$ in the homogeneous coordinate ring of $\Hilb^{n+k(k+1)/2}(\C^2)$, and the image of $\iota_k^*$ is given in Proposition \ref{prop:image of pullback}. The first isomorphism theorem therefore gives an identification 
    \[ \bigoplus_{m\geq 0}A^m_{\geq km}(n) \simeq \left(\bigoplus_{m\geq 0} A^m\left( n+\frac{k(k+1)}{2} \right)\right) \bigg/ I(X_k(n)). \]
    where the right-hand side is the homogeneous coordinate ring of $X_k(n) \subseteq \Hilb^{n+k(k+1)/2}(\C^2)$ as desired.
\end{proof}

We have similar results for the product map
\[ \iota_k\times\iota_{k+1}:\Hilb^n(\C^2\setminus\mathbf{0})\to \Hilb^{n+k(k+1)/2}(\C^2)\times \Hilb^{n+(k+1)(k+2)/2}(\C^2). \]

\begin{corollary}\label{cor:image of bigraded pullback}
    Pullback by $\iota_k\times \iota_{k+1}$ defines a ring map
    \begin{equation*}
         \bigoplus_{m_1\geq 0} A^{m_1}\left(n+\frac{k(k+1)}{2}\right)\otimes \bigoplus_{m_2\geq 0} A^{m_2}\left(n+\frac{(k+1)(k+2)}{2}\right) \to \bigoplus_{m_1,m_2\geq 0} A^{m_1+m_2}(n).
    \end{equation*}
    whose image is
    \[ \bigoplus_{m_1,m_2\geq 0} A^{m_1+m_2}_{\geq m_1k+m_2(k+1)}(n). \]
\end{corollary}

\begin{proof}
    The map $\iota_k\times\iota_{k+1}$ factors as the diagonal embedding into the locus
    \[ \Hilb^n(\C^2\setminus\{\mathbf{0}\})\times \Hilb^n(\C^2\setminus\{\mathbf{0}\}) \subseteq \Hilb^{n+k(k+1)/2}(\C^2)\times \Hilb^{n+(k+1)(k+2)/2}(\C^2) \]
    where  the factors are embedded as $\iota_k$ and $\iota_{k+1}$ separately. By Proposition \ref{prop:image of pullback}, the image of the restriction map to this product is
    \[ \bigoplus_{m_1,m_2\geq 0} A^{m_1}_{\geq m_1 k}(n) \otimes A^{m_2}_{\geq m_2 (k+1)}(n). \]
    Further restricting to the diagonal turns tensor products into ordinary products, and by Theorem \ref{thm:infinite chambers fg} we have
    \[ A^{m_1}_{\geq m_1 k}(n) \cdot A^{m_2}_{\geq m_2 (k+1)} = \left(A^{1}_{\geq k}(n)\right)^{m_1} \cdot \left(A^{1}_{\geq k+1}(n)\right)^{m_2} = A^{m_1+m_2}_{\geq m_1k+m_2(k+1)}(n) \]
    as desired.
\end{proof}

We can now prove Theorem \ref{thm:main2} using a similar argument to Theorem \ref{thm:main1}.

\begin{proof}[Proof of Theorem \ref{thm:main2}]
    By the same argument as in the proof of Theorem \ref{thm:main1}, the kernel of $(\iota_k\times \iota_{k+1})^*$ is the ideal of $X_{k,k+1}$ in the bigraded section ring of the product 
    \[ \Hilb^{n+k(k+1)/2}(\C^2)\times \Hilb^{n+(k+1)(k+2)/2}(\C^2). \]
    Corollary \ref{cor:image of bigraded pullback} and the First Isomorphism Theorem imply that the bigraded section ring of $X_{k,k+1}$ is 
    \[ \bigoplus_{m_1,m_2\geq 0} A^{m_1+m_2}_{\geq m_1k+m_2(k+1)}(n), \]
    as desired. 
\end{proof}

\section{Birational maps and base loci of line bundles on $\Hilb^n(\Bl_0 (\C^2))$}\label{sec:base locus}

In the previous section, we proved Theorems \ref{thm:main1} and \ref{thm:main2} identifying $X_k$ and $X_{k,k+1}$ as certain birational models of $\Hilb^n(\Bl_{\mathbf{0}}(\C^2))$. In this section, we explicitly describe the birational maps to these models, and their base loci.

\subsection{The rational maps to the models}

To describe the rational maps, we adapt a construction used in \cite{ABCH} for the Hilbert scheme of points on $\P^2$.

Fix an integer $k\geq0$. For each $Z\in \Hilb^n(\Bl_{\mathbf{0}}(\C^2))$, we have an exact sequence of sheaves on $\Bl_{\mathbf{0}}(\C^2)$
\[ 0\to \mathcal{I}_{Z}(-kE) \to \CO_{\Bl_{\mathbf{0}}(\C^2)}(-kE) \to \CO_Z(-kE) \to 0. \]
Taking global sections, we obtain  
\begin{equation*}
    I_{Z,k}:= H^0\left(\Bl_{\mathbf{0}}(\C^2), \, \mathcal{I}_Z(-kE)\right) \subseteq H^0\left(\Bl_{\mathbf{0}}(\C^2), \, \CO_{\Bl_{\mathbf{0}}(\C^2)}(-kE)\right) \simeq \C[x,y]_{\geq k}\subseteq \C[x,y].
\end{equation*} 

The content of the following proposition is that $I_{Z,k}\subseteq \C[x,y]$ is an ideal, and for generic $Z$ this ideal has colength $n+k(k+1)/2$.

\begin{proposition}
    The operation $Z\mapsto I_{Z,k}$ defines a rational map
    \[ \varphi_k:\Hilb^n(\Bl_{\mathbf{0}}(\C^2)) \dashrightarrow \Hilb^{n+k(k+1)/2}(\C^2) \]
    whose essential image is the Brill-Noether locus $X_k(n)$.
\end{proposition}

\begin{proof}
    Let $s_E$ denote the canonical section of $\CO_{\Bl_{\mathbf{0}}(\C^2)}(E)$ and $\pi:\Bl_{\mathbf{0}}(\C^2) \to \C^2$ the blowup. We may then describe $I_{Z,k}\subseteq \C[x,y]$ as the set of global functions $f$ on $\C^2$ vanishing to order $\geq k$ at the origin and such that the restriction of $\pi^*(f)/s_E^k$ to $Z$ is $0$. It is clear from this description that $I_{Z,k}\subseteq \C[x,y]$ is an ideal. 
    
    Now we show that the dimension of $\C[x,y]/I_{Z,k}$ is generically $n+k(k+1)/2$ for $Z$ on a dense open subset of $\Hilb^n(\Bl_{\mathbf{0}}(\C^2))$. Indeed, if the support of $Z$ is disjoint from $E$ we can push $Z$ forward to a subscheme $Z'\subseteq \C^2$ whose support is disjoint from the origin. In this case, $I_{Z,k}$ can be described as 
    \[ I_{Z,k} = (x,y)^k \cap I(Z') \subseteq \C[x,y]. \]
    This is the ideal of the disjoint union of the length $n$ subscheme $Z'$ with a nonreduced scheme supported at the origin of length $k(k+1)/2$. This shows that the operation $Z\mapsto I_{Z,k}$ defines a rational map to $\Hilb^{n+k(k+1)/2}(\C^2)$. Finally, the essential image of this rational map is the closure of the image of this open subset, which is $X_k(n)$ by definition.
\end{proof}

Now we describe the indeterminacy loci of the rational maps $\varphi_k$. Let $Y_\ell\subseteq \Hilb^n(\Bl_0 (\C^2))$ be the locus of subschemes $Z$ such that the length of the scheme theoretic intersection $Z\cap E$ is at least $\ell$. we have inclusions 
\[\Hilb^n(\Bl_0 (\C^2)) = Y_0\supsetneq Y_1 \supsetneq \cdots \supsetneq Y_n = \Hilb^n(E). \]
For $\ell>n$, $Y_\ell$ is empty.

\begin{lemma}
\label{lem: Yl irreducble}
For each $\ell=0,1,\dots,n$, $Y_\ell$ is an irreducible  subvariety of $\Hilb^n(\Bl_0 (\C^2))$ of codimension $\ell$.
\end{lemma}

\begin{proof}
Without loss of generality, consider the local model where the pair $E\subset \Bl_0 (\C^2)$ is replaced by $\{y=0\}\subset \C^2$. The length of the scheme theoretic intersection $\Spec(I)\cap E$ equals the dimension of cokernel (or, equivalently, kernel) of $y$ acting on $\C[x,y]/I$. We get 
$$
Y_{\ell}^{\mathrm{loc}}=\{\C[x,y]\supset I: \dim \Ker\  y|_{\C[x,y]/I}\ge \ell\}\subset \Hilb^n(\C^2),
$$
and another variety 
$$
\widetilde{Y}_{\ell}^{\mathrm{loc}}=\{\C[x,y]\supset J\supset I: yJ\subset I,\dim J/I=\ell\}\subset \Hilb^{n-\ell,n}(\C^2).
$$
Note that $yJ\subset I$ is equivalent to $J/I\subset \Ker\  y|_{\C[x,y]/I}$.
Similarly to \cite[Theorem 4.1.6]{CGM} one can check that $\widetilde{Y}_{\ell}^{\mathrm{loc}}$ is smooth and connected, so it is irreducible. The natural projection $\widetilde{Y}_{\ell}^{\mathrm{loc}}\to Y_{\ell}^{\mathrm{loc}}$
is surjective, 
so $Y_{\ell}^{\mathrm{loc}}$ is irreducible as well. On the other hand,  $Y_{\ell}^{\mathrm{loc}}$ has an open subset where all points are distinct, and $\ell$ points are on the line $\{y=0\}$ - this has dimension $2n-\ell$, so $\dim Y_{\ell}^{\mathrm{loc}}=2n-\ell$.
\end{proof}

\begin{proposition}\label{prop:ind locus}
    For all $k\geq 0$, the indeterminacy locus of $\varphi_k$ is $Y_{k+1}$. In other words, we have
    \[ \dim_{\C}\left(\C[x,y]/I_{Z,k}\right) = n + \frac{k(k+1)}2 \]
    if and only if $Z\notin Y_{k+1}$.
\end{proposition}

In particular, Proposition \ref{prop:ind locus} implies that $\varphi_k$ is regular if $k\geq n$, since in this case $Y_{k+1}$ is empty.

\begin{proof}
    First, note that we have $I_{Z,k}\subseteq \C[x,y]_{\geq k}$ by definition, so 
    \[ \dim_{\C}\left(\C[x,y]/I_{Z,k}\right) = n + \frac{k(k+1)}2 \]
    is equivalent to 
    \[ \dim_{\C}\left(\C[x,y]_{\geq k}/I_{Z,k}\right) = n. \]
    We now describe $\C[x,y]_{\geq k}/I_{Z,k}$ in local coordinates. Consider the open subset $U\subseteq \Bl_{\mathbf{0}}(\C^2)$ given by the complement of the strict transform of the line $x=0$ in $\C^2$. Possibly choosing a new coordinate system on $\C^2$, we may assume without loss of generality that $Z$ is contained in $U$. This open set $U$ is isomorphic to $\C^2$ with coordinates $(x,u)$ where the map 
    \[ U\hookrightarrow \Bl_{\mathbf{0}}(\C^2) \to \C^2\]
    is given by $(x,u)\mapsto (x,xu)$, and $E\cap U \subseteq U$ is defined by $x=0$. This shows that, under the identification of global sections of $\CO_{\Bl_{\mathbf{0}}(\C^2)}(-kE)$ with $\C[x,y]_{\geq k}$, the restriction of any global section $f(x,y)\in \C[x,y]_{\geq k}$ to $U$ is given in coordinates by $f(x,xu)/x^k\in \C[x,u]$. The image of a monomial $x^i y^j$ with $i+j\geq k$ under this map is $x^{i+j-k}u^j$. This implies that the image of the restriction map of global sections of $\CO_{\Bl_{\mathbf{0}}(\C^2)}(-kE)$ to $U$ is the subspace
    \[ W_k := \mathrm{span}\{ x^a u^b \, | \, b\leq a+k \}\subseteq \C[x,u]. \]
    
    We have taken $Z\subseteq U$, so there is a corresponding ideal $I(Z)\subseteq \C[x,u]$ of colength $n$. In our local coordinates, $I_{Z,k}\subseteq \C[x,y]_{\geq k}$ is identified with 
    \[ W_k \cap I(Z) \subseteq W_k. \]
    Equivalently, $W_k \cap I(Z)$ is the kernel of the composition of linear maps
    \[ W_k \hookrightarrow \C[x,u] \to \C[x,u]/I(Z). \]
    Since $I(Z)\subseteq \C[x,u]$ has colength $n$, we find that $I_{Z,k}\subseteq \C[x,y]_{\geq k}$ has colength $n$ if and only if $W_k$ spans $\C[x,u]/I(Z)$. 

    Now consider $I(Z\cap E) = I(Z) +(x) \subseteq \C[x,u]$. We have a commutative diagram of linear maps
    \begin{equation}\label{eq:diagram of sections} \begin{tikzcd} 
    W_k \arrow[r] \arrow[d] & \C[x,u]/I(Z) \arrow[d] \\
    W_k/(W_k\cap (x)) \arrow[r] & \C[x,u]/(I(Z)+(x))
    \end{tikzcd}\end{equation}
    coming from the inclusion $W_k\subseteq \C[x,u]$ and the natural quotient maps. Observe that the dimension of $\C[x,u]/(I(Z)+(x))$ is equal to the length of the scheme theoretic intersection $Z\cap E$, and $W_k/(W_k\cap (x))$ is spanned by $1,u,\dots,u^k$. Furthermore, note that both vertical maps in \eqref{eq:diagram of sections} are surjective.

    We need to show that the top map of \eqref{eq:diagram of sections} is surjective if and only if \[ \dim(\C[x,u]/(I(Z)+(x)))\leq k+1.\]

    If $W_k$ spans $\C[x,u]$ modulo $I(Z)$, so that the top map in \eqref{eq:diagram of sections} is surjective, then the bottom map in \eqref{eq:diagram of sections} is also surjective. Hence $\C[x,u]/(I(Z)+(x))$ is spanned by $1,u,\dots,u^k$ and therefore has dimension at most $k+1$ as desired.

    On the other hand, if $\dim(\C[x,u]/(I(Z)+(x)))\leq k+1,$ then it is necessarily spanned by $1,u,\dots,u^k$. This is because every finite-dimensional quotient $\C[x,u]/J$ has some basis of monomials $x^iy^j$ that is closed under divisibility, and all monomials divisible by $x$ are zero in $\C[x,u]/(I(Z)+(x))$.

    Finally, we will show that in this case every monomial $x^iu^j$ is equivalent to a linear combination of monomials in $W_k$ up to $I(Z)$. Indeed, since $1,u,\dots,u^k$ span $\C[x,u]/(I(Z)+(x))$, for each $j$ there is a relation $u^j -f(u) \in I(Z)+(x)$ for some polynomial $f(u)$ of degree at most $k$. Note that this polynomial is in $\C[u] \cap \left(I(Z)+(x) \right) = \C[u]\cap I(Z)$, so we in fact have $u^j -f(u) \in I(Z)$. Multiplying by $x^i$, we then also have $x^iu^j-x^if(u) \in I(Z)$. But all the monomials in $x^if(u)$ are in $W_k$, as $W_k$ is closed under multiplication by $x$ and contains $1,u,\dots,u^k$. This shows that $W_k$ spans $\C[x,u]/I(Z)$, which completes the proof.
\end{proof}

Now we state the main theorem of this section, determining the base loci of the line bundles $\CO(k,m)$.

\begin{theorem}\label{thm:base locus}
    For $k,m\geq 0$, the base locus of $\CO(k,m)$ contains $Y_\ell$ if and only if $k<m(\ell-1)$.
\end{theorem}

Before the proof of Theorem \ref{thm:base locus} we need a lemma.

\begin{lemma}\label{lem:pullback to P^l}
    For distinct points $p_1,\dots,p_{n-\ell}\in\C^2\setminus \mathbf{0}$, we have a closed subvariety $\P^\ell\simeq Y_\ell' \subseteq \Hilb^n(\Bl_0 (\C^2))$ consisting of the subschemes that are the reduced union of the points $p_1,\dots,p_{n-\ell}$ with a closed subscheme of $E$ of length $\ell$. The pullback of $\CO(k,m)$ to $Y_\ell'$ is identified with the line bundle $\CO(k-m(\ell-1))$ on $\P^\ell$.
\end{lemma}

\begin{proof}
    The subvariety $Y_\ell'$ can be identified with the Hilbert scheme of $\ell$ points on $E\simeq \P^1$ which is well-known to be isomorphic to $\P^\ell$.
    
    We compute the intersection of $Y_\ell'$ with divisors $B,D\subseteq \Hilb^n(\Bl_0 (\C^2))$ where $B$ is the locus of nonreduced schemes and $D$ defined as follows: Fix a generic line $L$ in $\C^2$ through the origin and let $\widetilde{L}\subseteq \Bl_{\mathbf{0}}(\C^2)$ denote the strict transform of $L$. The divisor $D$ is defined to be the locus of subschemes with nontrivial intersection with $\widetilde{L}$. These divisors are representatives for the first Chern classes of $\CO(0,-2)$ and $\CO(1,0)$ respectively.
    
    Both divisors $B$ and $D$ intersect $Y_\ell'$ transversely. The intersection of $Y_\ell'$ with $B$ is identified with the discriminant locus of $\P^\ell$ which has degree $2(\ell-1)$. The intersection of $Y_\ell'$ with $D$ is the locus of length $\ell$ subschemes of $E\simeq \P^1$ containing a fixed point, the intersection of $E$ with $\widetilde{L}$. This is identified with a hyperplane in $\P^\ell$. This identifies the pullbacks of $\CO(0,-2)$ and $\CO(1,0)$ with $\CO(2(\ell-1))$ and $\CO(1)$ respectively. The claimed expression for the pullback of an arbitrary line bundle $\CO(k,m)$ follows from linearity.
\end{proof}

\begin{proof}[Proof of Theorem \ref{thm:base locus}]

    Take $p_1,\dots,p_{n-\ell}\in \C^2\setminus \mathbf{0}$ distinct and let $Y_\ell'$ be as in Lemma \ref{lem:pullback to P^l}. As $p_1,\dots,p_{n-\ell}$ vary, these $Y_\ell'$'s cover a dense open subset of $Y_\ell$. Since the base locus of any line bundle is closed, it suffices to show that for any fixed $p_1,\dots,p_{n-\ell}$, $Y_\ell'$ is in the base locus of $\CO(k,m)$ if and only if $k<m(\ell-1)$. We have by Lemma \ref{lem:pullback to P^l} that the pullback of $\CO(k,m)$ to $Y_\ell'\simeq \P^\ell$ is identified with $\CO(k-m(\ell-1))$. 
    
    On the one hand, if $k<m(\ell-1)$, the restricted line bundle on $Y_\ell'$ has no nonzero sections. This implies that $Y_\ell'$ is contained in the base locus of $\CO(k,m)$ as desired.

    Now suppose $k\geq m(\ell-1)$ so that the restriction of $\CO(k,m)$ to $Y_\ell'$ does have nonzero sections. It suffices to show that there is a section of $\CO(k,m)$ whose restriction to $Y_\ell'$ is not identically zero. We claim that it suffices to exhibit such a section $s$ for the line bundle $\CO(\ell-1,1)$. Indeed, we may then exhibit such a section of $\CO(k,m(\ell-1))$ for any $k$ and $m$ with $k\geq m(\ell-1)$ by multiplying the section $s^m$ of $\CO(m(\ell-1),m)$ with an appropriately chosen section of some power of the basepoint-free line bundle $\CO(1,0)$.

    Since $\CO(\ell-1,1)$ restricts to the trivial line bundle $\CO$ on $Y_\ell'\simeq \P^\ell$ by Lemma \ref{lem:pullback to P^l}, the restriction of any section $s$ to $Y_\ell'$ is constant. We therefore need only find a section $s\in \CO(\ell-1,1)$ that is nonzero at some point of $Y_\ell'$. We claim that any  $s\in A^1_{\geq \ell-1}$ given as a determinant $\Delta_S$ for some set $S\subseteq \N^2$ containing the $\ell$ points $(0,\ell-1),(1,\ell-2),\dots,(\ell-1,0)$ will suffice. 

    To prove this claim, let us work in the same local coordinate system $(x,u)$ with $y=xu$ as in the proof of Proposition \ref{prop:ind locus}, where the exceptional curve is defined by $x=0$. For any $S$ as above, we can write $\Delta_S(x,y)$ in this coordinate system as
    \[\Delta_{S'}(x,u) = \det\left(x_i^{a_j'}u_i^{b_j'}\right) \]
    where 
    \[ S' = \{ (a_1',b_1'),\dots, (a_n',b_n') \}\]
    is the image of $S$ under the map $(a,b)\mapsto (a+b-(\ell-1),b)$. By assumption on $S$, $S'$ contains the points $(0,0),(0,1),\dots,(0,\ell-1)$, and all other $(a',b')\in S'$ have $a'>0$. 
    A general point of $Y_\ell'$ consists of the distinct points $p_1,\dots,p_{n-\ell}$ plus $\ell$ distinct points on the exceptional curve written in coordinates as $(0,u_1),\dots,(0,u_{\ell})$.
    
    At such a point, the matrix $\det\left(x_i^{a_j'}u_i^{b_j'}\right)$ is block upper triangular and hence $\Delta_{S'}$ splits into the product with the Vandermonde determinant
    \[ \Delta_{S'}(x,u) = \Delta_{S''}(x,u) \prod_{1\leq i < j \leq \ell} (u_i-u_j) \]
    where $\Delta_{S''}(x,u)$ is the analogous determinant for $S'' = S' \setminus \{(0,0),(0,1),\dots,(0,\ell-1) \}$ evaluated at the $(x,u)$-coordinates of the points $p_1,\dots,p_{n-\ell}$. This is nonzero for generic $p_1,\dots,p_{n-\ell}$, as desired.
\end{proof}

\section{Newton-Okounkov bodies and wall-crossing }\label{sec:NObodies}

In this section, we show that the wall and chamber structure on $\R^2_{\geq 0}$ corresponding to the various birational models of $\Hilb^n(\Bl_{\mathbf{0}}(\C^2))$ can be detected combinatorially using the Newton-Okounkov bodies computed in \cite{CGOT}. 

\begin{definition}\label{def:Delta(m,k)}
    For $k,m\geq 0$, let $\Delta(k,m)\subseteq\R^{2n}$, in coordinates $(a_1,\dots,a_n,b_1,\dots,b_n)$, be the subset defined by:
    \begin{enumerate}
        \item $0\leq a_1\leq a_2 \leq \cdots \leq a_n$, and
        \item for each $j=1,\dots,n,$ we have $b_j\geq \max\{k-a_j,0\} + \sum_{i=1}^{j-1}\max\{m-(a_j-a_i),0\}.$
    \end{enumerate}
\end{definition}

This convex polyhedron $\Delta(k,m)\subseteq \R^{2n}$ is the \textit{Newton-Okounkov body} of the corresponding line bundle on $\Hilb^n(\Bl_{\mathbf{0}}(\C^2))$ with respect to a certain valuation \cite{CGOT}. Since there is an identification of the sections of line bundles between the model $X_k$ or $X_{k,k+1}$ and the corresponding line bundle on $\Hilb^n(\Bl_{\mathbf{0}}(\C^2))$, one may also interpret $\Delta(k,m)\subseteq \R^{2n}$ as a Newton-Okounkov body for the corresponding birational model.

\begin{corollary}
    For any $k\geq 0$ and $m>0$, $\Delta(mk,m)$ is the Newton-Okounkov body of the line bundle $\CO(m)$ on $X_k$ with respect to the trailing term valuations on sections used in \cite{CGOT}. Similarly, for any $m_1,m_2\geq 0$, $\Delta(m_1k+m_2(k+1),m_1+m_2)$ is the Newton-Okounkov body of the line bundle $\CO(m_1,m_2)$ on $X_{k,k+1}$ with respect to the same valuation.
\end{corollary}

Our analysis of the discrete sets $\mathcal{P}(k,m)$ immediately gives several structural results for these polyhedra. First, the convex bodies vary linearly in each chamber we have found in the following sense.

\begin{proposition}\label{prop:NO body linearity}
    \begin{enumerate}
        \item For any integer $k\geq 0$ and real numbers $t_1,t_2>0$, there is an equality 
            \[ \Delta(t_1k+t_2(k+1),t_1+t_2) = t_1 \Delta(k,1)+ t_2 \Delta(k+1,1), \]
        where the right hand side is the Minkowski sum of dilations of polyhedra.
        \item For any real numbers $t_1,t_2>0$, there is an equality 
            \[ \Delta(t_1(n-1)+t_2,t_1) = t_1 \Delta(n-1,1)+t_2 \Delta(1,0). \]
    \end{enumerate}
\end{proposition}

\begin{proof}
    For any integers $k,m$, $\mathcal{P}(k,m)\subseteq \Delta(k,m)$ contains all the interior integer points of $\Delta(k,m)$ since the only additional constraints in the definition of $\mathcal{P}(k,m)$ are conditional on an equality $a_i=a_{i+1}$. Now for any integer $t\geq 0$, $\Delta(tk,tm) = t \Delta(k,m)$ so the same argument implies that 
    \[ \frac{1}{t}\mathcal{P}(tk,tm) \subseteq \Delta(k,m) \]
    contains all the interior $\frac{1}{t}\Z$ points of $\Delta(k,m).$ These claims therefore follow from the corresponding facts about the discrete sets, Proposition \ref{prop:decomp} and Lemma \ref{lem: m=1 decomp}
\end{proof}

Proposition \ref{prop:NO body linearity} says that the Newton-Okounkov bodies vary linearly on the chambers dividing line bundles on $\Hilb^n(\Bl_{\mathbf{0}}(\C^2))$ corresponding to different birational models and base loci. One can also show that the Newton-Okounkov bodies do not vary linearly on any coarser chamber decomposition. One way to show this is to show that the combinatorial structure of the polyhedron $\Delta(k,m)$ changes when crossing over each wall. We do this by introducing a subset of $\Delta(k,m)$ that corresponds to the possible base loci $Y_\ell$.

\begin{definition}
    Let $\mathcal{Y}_\ell \subseteq \Delta(k,m)$ denote the subset of all points $(a_1,\dots,a_n,b_1,\dots,b_n)\in\Delta(k,m)$ such that at least $\ell$ of the points $(a_1,b_1),\dots,(a_n,b_n)$ satisfy $a_i+b_i = k$.
\end{definition}

Informally, we regard $\Delta(k,m)$ as a combinatorial model for the birational model of $\Hilb^n(\Bl_{\mathbf{0}}(\C^2))$ corresponding to $\CO(k,m)$, and $\mathcal{Y}_\ell \subseteq \Delta(k,m)$ as corresponding to the image of $Y_\ell$ in this model. We therefore consider the following as a combinatorial shadow of the base locus decomposition of $\Hilb^n(\Bl_{\mathbf{0}}(\C^2))$.

\begin{proposition}\label{prop:base locus analogue for NO body}
    For $k,m\geq 0$, the subset $\mathcal{Y}_\ell\subseteq \Delta(k,m)$ is empty if and only if $k<m(\ell-1)$ or $\ell>n$.
\end{proposition}

\begin{proof}
    Let $(a_1,\dots,a_n,b_1,\dots,b_n)\in \Delta(k,m)$ and let $j$ be such that $a_j+b_j=k$. Then by the definition of $\Delta(k,m)$ we have 
    \[ k-a_j\geq 0,\, \text{ and } a_j-a_i \geq m \] 
    for all $i<j$. Suppose $\mathcal{Y}_\ell$ is nonempty and take $(a_1,\dots,a_n,b_1,\dots,b_n)$ and $i_1<\cdots<i_{\ell}$ such that $a_{i_1}+b_{i_1} = \cdots= a_{i_\ell}+b_{i_\ell}$. We clearly have $\ell\leq n$ as $\ell$ points out of the $n$ total satisfy the condition $a_i+b_i=k$. Furthermore, the previous analysis implies that $0\leq a_{i_1} \leq \cdots \leq a_{i_\ell}\leq k$ and the difference between consecutive terms is at least $m$. This implies that $k\geq m(\ell-1)$.
    One the other hand, if $k\geq m(\ell-1)$ and $\ell\leq n$, one checks that the point $\left( a_1,\dots,a_n,b_1,\dots,b_n\right)$ defined by 
    \[ a_i = (i-1)m, \, \text{ and } b_i = \max\{ k-(i-1)m,0 \} \]
    lies in $\mathcal{Y}_\ell$.
\end{proof}

\section{Enumerative character formulas}
\label{sec: character}

In this section we give two combinatorial formulas for the $(\C^{\times})^2$-equivariant Euler characteristics of line bundles on $\Hilb^n(\Bl_{\mathbf{0}}(\C^2))$ and the models $X_k$ and $X_{k,k+1}$. 

On the one hand, by Theorem \ref{thm:HS lead terms} we can describe the basis in the space of sections by counting points in the appropriate polytope $\CP$, see Lemma \ref{lem:  char sections} below. A line bundle in a chamber spanned by the vectors $(k,1)$ and $(k+1,1)$ corresponds to an ample line bundle  on $X_{k,k+1}$, and we prove in Corollary \ref{cor: euler char} that its equivariant Euler characteristic is also determined by $\CP$. 

On the other hand, we use equivariant localization on $X_{k,k+1}$ to compute the equivariant Euler characteristic as a sum over fixed points of torus action. Comparing these two formulas leads to a series of highly nontrivial combinatorial identities, one for each $k$. For the very first chamber with $k=0$ the corresponding identities were described in detail in \cite[Section 4.2]{CGOT}, and the other ones can be thought of as deep generalizations of these.

\subsection{Characters for global sections}

We abuse notation by writing $\CO(d)$ for the restriction of $\CO(d)$ from $\Hilb^{n+k(k+1)/2}(\C^2)$ to $X_k$, and similarly $\CO(d_1,d_2)$ for restriction of the tensor product of pullbacks of $\CO(d_1)$ and $\CO(d_2)$ from the factors $\Hilb^{n+k(k+1)/2}(\C^2)\times \Hilb^{n+(k+1)(k+2)/2}(\C^2)$ to $X_{k,k+1}$. We have seen that for $d>0$ we have
\[ H^0(X_k,\CO(d)) \simeq H^0(\Hilb^n(\Bl_{\mathbf{0}}(\C^2)),\CO(dk,d))\simeq A^d_{\geq dk}, \]
and for $d_1,d_2>0$ we have
\begin{multline}
\label{eq:comb-H0}
  H^0(X_{k,k+1},\CO(d_1,d_2)) \simeq A^{d_1+d_2}_{\geq d_1k+d_2 (k+1)} \simeq\\ H^0(\Hilb^n(\Bl_{\mathbf{0}}(\C^2)),\CO(d_1k+d_2 (k+1),d_1+d_2)).
\end{multline}

Theorem \ref{thm:HS lead terms} implies the following:

\begin{lemma}
\label{lem:  char sections}
For $d_1,d_2>0$ we have
$$
\ch H^0(X_{k,k+1},\CO(d_1,d_2))=\sum_{(\aaa,\bbb)\in \CP}q^{|\aaa|}t^{|\bbb|}
$$
where $\CP=\CP(d_1k+d_2(k+1),d_1+d_2)$.
\end{lemma}

Thus we immediately obtain that the line bundle \(\CO(d_1,d_2)\) is ample on $X_{k,k+1}$ if \(d_1,d_2>0\).


\begin{example}
\label{ex:weightedcount}
    Let us compute the $(\C^{\times})^2$-character of $A^1_{\geq 2}(3)$ as a weighted sum over $\mathcal{P}(2,1)$. This is the collection of points of triples of integer points $(a_1,b_1)<(a_2,b_2)<(a_3<b_3)$ in lexicographic order with $a_i+b_i\geq 2$ for each $i$. We partition $\mathcal{P}(2,1)$ into subsets depending on the tuple $(a_1,a_2,a_3)$ and compute the weighted sum of each separately.
    \begin{enumerate}
        \item $a_1 = a_2 = a_3$. The weighted sum over such tuples is
        \[ \frac{q^3}{(1-q)(1-q^2)(1-q^3)}\cdot \left( q^6+t^3 q^3+\frac{t^6}{1-t^3} \right), \]
        where the terms in the second factor correspond to the cases $a_1 = 0,1,$ and $a_1\geq 2$.
        \item $a_1=a_2<a_3$. The weighted sum is
        \[ \frac{q}{(1-q)(1-q^2)}\cdot \frac{t}{1-q}\cdot \left( q^5+\frac{tq^4+t^3 q^2}{1-t} + \frac{t^6}{(1-t)(1-t^3)} \right). \]
        \item $a_1<a_2=a_3$. The weighted sum is
        \[ \frac{1}{1-q}\cdot \frac{t^2q}{(1-q)(1-q^2)}\cdot \left( q^4+\frac{t^2q^2+t^3 q}{1-t^2} + \frac{t^6}{(1-t^2)(1-t^3)} \right). \]
        \item $a_1<a_2<a_3$. The weighted sum is
        \[ \frac{1}{1-q}\cdot \frac{t}{1-q}\cdot \frac{t^2}{1-q}\cdot \left( \frac{q^3}{1-t}+\frac{t^2q^2+t^3q}{(1-t)(1-t^2)} + \frac{t^6}{(1-t)(1-t^2)(1-t^3)} \right). \]
        Summing these $4$ rational functions gives the $t,q$-weight generating function for $\mathcal{P}(2,1)$.   
    \end{enumerate}
\end{example}

\subsection{Homology vanishing}

We prove highed homology vanishing for ample line bundles on $X_{k,k+1}$ using the idea of Frobenius splitting \cite{FrobBook}.

\begin{proposition}
  The variety \(X_{k,k+1}\) is Frobenius split.
\end{proposition}
\begin{proof}
  The affine plane \(\C^2\) is Frobenius split and hence by \cite[Theorem 7.5.2]{FrobBook}
  \(\Hilb^n(\C^2)\) is Frobenius split as well. 
  Thus \(\Hilb^{n+\binom{k+1}{2}}(\C^2)\times \Hilb^{n+\binom{k+2}{2}}(\C^2)\) is Frobenius split and we
  show below that this splitting restricts to a Frobenius splitting of \(X_{k,k+1}\). In other words, \(X_{k,k+1}\) and \(\Hilb^{n+\binom{k+1}{2}}(\C^2)\times \Hilb^{n+\binom{k+2}{2}}(\C^2)\) have compatible Frobenius splittings.

  Let us denote by \(\phi_n\in \Hom\left(F_*\mathcal{O}_{\Hilb^n(\C^2)},\mathcal{O}_{\Hilb^n(\C^2)}\right)\) the Frobenius splitting from the above.
  The statement would follow from \cite[Lemma 1.1.7(ii)]{FrobBook} if we find an open subset \(U\subset \Hilb^{n+\binom{k+1}{2}}(\C^2)\times \Hilb^{n+\binom{k+2}{2}}(\C^2)\) such that \(U\cap X_{k,k+1}\) is dense in \(X_{k,k+1}\) and the Frobenius splitting
  \(\phi_{n+\binom{k+1}{2}}\boxtimes \phi_{n+\binom{k+2}{2}}\) restricts to a Frobenius splitting of \(U\cap X_{k,k+1}\).
  Let us define the sought after open set \(U\) as product \(U_1\times U_2\) where \(U_1\subset \Hilb^{n+\binom{k+1}{2}}(\C^2)\),
  \(U_2\subset \Hilb^{n+\binom{k+2}{2}}(\C^2)\) are defined as follows. An ideal \(I\) is in \(U_1\), \(U_2\) respectively, if
  the length of \((\C[x,y]/I)\otimes \C[x,y]_{(x,y)}\) is at most \(k(k+1)/2\) and \((k+1)(k+2)/2\) respectively.

  The above mentioned conditions of \cite[Lemma 1.1.7(ii)]{FrobBook} are satisfied because on the one hand \(U\cap X_{k,k+1}=\Hilb^n(\C^2\setminus (0,0))\).
  On the other hand the Frobenius splitting \(\phi_n\) is obtained from the canonical Frobenius splitting of \(\Sym^n(\C^2)\) by the pull-back along the Hilbert-Chow map \(\HC_n\), crepantness of \(\HC_n\) is crucial here. The canonical Frobenius splitting of
  \(\Sym^n(\C^2)\) is the \(S_n\)-equivariant splitting \(\phi^{\boxtimes n}\) where \(\phi\) is the canonical splitting of
  \(\C^2\).

  Hence the \(S_{n+\binom{k+1}{2}}\times S_{n+\binom{k+2}{2}}\)-equivariant splitting \(\phi^{\boxtimes n+\binom{k+1}{2}}\boxtimes \phi^{\boxtimes n+\binom{k+2}{2}}\) restricts to the \(S_n\)-equivariant splitting \(\phi^{\boxtimes n}\) on 
  \[\left[\HC_{n+\binom{k+1}{2}}\times \HC_{n+\binom{k+2}{2}}(U)\right]\cap X_{k,k+1}=\Sym^n(\C^2\setminus (0,0)).
  \] 
  Moreover, the restriction of \(\HC_{n+\binom{k+1}{2}}\times \HC_{n+\binom{k+2}{2}}\) to $U$ coincides with  the Hilbert-Chow map
  \(\HC_n\) for \(\C^2\setminus (0,0)\):
  \[\HC^*_{n+\binom{k+1}{2}}\left(\phi^{\boxtimes n+\binom{k+1}{2} }\right)\boxtimes \HC^*_{n+\binom{k+2}{2}}\left(\phi^{\boxtimes n+\binom{k+2}{2}}\right)\mid_{U\cap X_{k,k+1}}=
  \HC^*_{n}\left(\phi^{\boxtimes n }\right).\]
\end{proof}

\begin{corollary}
  For \(d_1,d_2>0\) we have \( H^{>0}(X_{k,k+1},\CO(d_1,d_2))=0\).
\end{corollary}
\begin{proof} Since \(\CO(d_1,d_2)\) is ample for \(d_1,d_2>0\) the statement follows from \cite[Theorem 1.2.8]{FrobBook}.
  \end{proof}

\begin{corollary}
\label{cor: euler char}
For $d_1,d_2>0$ we have
$$
\chi_{(\C^{\times})^2}(X_{k,k+1},\CO(d_1,d_2))=\ch H^0(X_{k,k+1},\CO(d_1,d_2))=\sum_{(\aaa,\bbb)\in \CP}q^{|\aaa|}t^{|\bbb|}
$$
where $\CP=\CP(d_1k+d_2(k+1),d_1+d_2)$.
\end{corollary}

\subsection{Localization}

We can compute the Euler characteristics of line bundles on $X_{k,k+1}$ using localization. Recall that $X_{k,k+1}$ is a smooth closed subscheme of the (singular) nested Hilbert scheme
$$
X_{k,k+1}\subset \Hilb^{n+\binom{k+1}{2},n+\binom{k+2}{2}}.
$$
By Proposition \ref{prop: X as BN} we can realize $X_{k,k+1}$ as the set of pairs of ideals $\C[x,y]\supset I\supset J$ such that 
$\dim \C[x,y]/I=n+\binom{k+1}{2},\dim \C[x,y]/J=n+\binom{k+2}{2}$ and 
$$
xI\subset J,\ yI\subset J.
$$
Equivalently, we can consider  two vector spaces:
$\C[x,y]/I=V_I, \C[x,y]/J=V_J$.
We have the following diagram:
\begin{equation}
\label{eq: quiver}
\begin{tikzcd}
\C \arrow{r}{\iota} & V_J \arrow[twoheadrightarrow]{r}{\pi} & V_I \arrow[bend left]{l}{A} \arrow[bend right,swap]{l}{B}  
\end{tikzcd}
\end{equation}
where $\pi:V_J\to V_I$ is the natural projection, $\iota(1)=1$ and $A,B$ are the operators of multiplication by $x,y$ respectively.  Note that $\ker \pi=I/J$. We also add a formal map $0=\jmath:V_I\to \C$. The operators $A,B$ satisfy the {\em moment map equation}
$$
\mu(A,B)=A\pi B-B\pi A+\iota\jmath=0
$$
which expresses the fact that $x$ and $y$ commute, and {\em stability condition} which says that $A\pi, B\pi$ applied to the image of $\jmath$ generate $V_J$. One can check \cite{GGS} that $X_{k,k+1}$ is isomorphic to the moduli space of diagrams \eqref{eq: quiver} satisfying the moment map equation and stability condition, up to changes of basis in $V_J$ and $V_I$.  

Two-dimensional torus $(\C^{\times})^2$ acts on $\C^2$ by rescaling the coordinates, this action extends to the nested Hilbert scheme and to $X_{k,k+1}$. In terms of the diagram \eqref{eq: quiver}, the action does not change $\jmath$ and $\pi$ and rescales $A$ by $q^{-1}$ and $B$ by $t^{-1}.$  
The fixed points of the torus action are labeled by the pairs of Young diagrams $\lambda\subset \mu$ such that $|\lambda|=n+\binom{k+1}{2}$, $|\mu|=n+\binom{k+2}{2}$ and no two boxes in the skew diagram $\mu\setminus \lambda$ are next to each other.  Let us denote the set of such pairs of diagrams by \(D(n,k)\):
\[D(n,k)=(X_{k,k+1})^{(\C^\times)^2}.\]

We the following description of the tangent space to $X_{k,k+1}$:

\begin{theorem}\cite[Theorem 2.23]{GGS}
Consider the $(\C^{\times})^2$-equivariant complex 
\begin{multline}
\label{eq: tangent complex}
\mathrm{Lie}(P)\rightarrow \\q^{-1}\Hom(V_I,V_J)\oplus t^{-1}\Hom(V_{I},V_J)\oplus \Hom(\C,V_J)\oplus (qt)^{-1}\Hom(V_I,\C)\xrightarrow{d\mu}\\
 (qt)^{-1}\Hom(V_I,V_J)
\end{multline}
where $P$ is the group of block-triangular matrices with blocks of size $(n+\binom{k+1}{2},k+1)$ and $\mathrm{Lie}(P)$ its Lie algebra; the first arrow is given by the differential of the action of $P$ and the second arrow is the differential of $\mu$. Then:
\begin{enumerate}
\item[(a)] On stable locus the first map is injective and the second map is surjective.
\item[(b)] The tangent space to $X_{k,k+1}$ is given by the middle cohomology of \eqref{eq: tangent complex}.
\end{enumerate}
\end{theorem}

\begin{corollary}
At a fixed point $(\lambda,\mu)$, the character of the cotangent space $\Omega_{\lambda,\mu}X_{k,k+1}$ equals:
\begin{equation}
\label{eq: localization weight}
-(B_{\mu}-B_{\lambda})(B_{\mu}^{\ast}-B_{\lambda}^{\ast})-(1-q)(1-t)B_{\lambda}B_{\mu}^{\ast}+qtB_{\lambda}+B_{\mu}^{\ast}.
\end{equation}
where 
$B_{\lambda}=\sum_{\sq\in \lambda}\sq,
B_{\lambda}^{\ast}=\sum_{\sq\in \lambda}\sq^{-1}$ and similarly for $B_{\mu},B_{\mu}^{\ast}$.
\end{corollary}

\begin{proof}
Let $\CV_I,\CV_J$ denote the tautological bundles on $X_{k,k+1}$ with fibers $V_I,V_J$ respectively. At a fixed point the characters of $\CV_I$ and $\CV_J$ are respectively given by $B_{\lambda}$ and $B_{\mu}$. 
Then
$$
\Omega X_{k,k+1}=-\mathrm{Lie}(P)^{\ast}+
 (q+t-qt)\CV_I\CV_J^{\ast}+qt\CV_I+\CV_J^{\ast}.
$$
We have 
$$
\ch\mathrm{Lie}(P)\simeq  (\CV_J-\CV_I)(\CV_J^{\ast}-\CV_I^{\ast})+ \CV_I^{\ast}\CV_J,
$$
so
$$
\ch\mathrm{Lie}(P)^{\ast}\simeq (\CV_J-\CV_I)(\CV_J^{\ast}-\CV_I^{\ast})\oplus \CV_I\CV_J^{\ast},
$$
and the result follows. 
\end{proof}

\begin{example}
For $n=2, k=0$ we get $|\lambda|=2$, $|\mu|=3$. For $\lambda=(2),\mu=(3)$ we get
$$
B_{\lambda}=1+q,\ B_{\mu}=1+q+q^2
$$
and \eqref{eq: localization weight} equals
$$
-q^2q^{-2}-(1-q)(1-t)(1+q)(1+q^{-1}+q^{-2})+qt(1+q)+(1+q^{-1}+q^{-2})=q+q^2+\frac{t}{q}+\frac{t}{q^2}.
$$
\end{example}

\begin{lemma}
\label{lem: triangular correction}
The $(\C^{\times})^2$-character of $\det\C[x,y]/(x,y)^k$ equals $(qt)^{\binom{k+1}{3}}$.
\end{lemma}

\begin{proof}
Clearly, the exponents of $q$ and $t$ are the same, and both are equal to
$$
(0+\ldots+(k-1))+(0+\ldots+(k-2))+\ldots+0=$$
$$
\binom{k}{2}+\binom{k-1}{2}+\ldots+\binom{1}{2}=\binom{k+1}{3}.
$$
\end{proof}

For a given Young diagram \(\lambda\) we denote by \(c_{q,t}(\lambda)\) the product of the \(q,t\)-contents of the diagram:
\(c_{q,t}(\lambda)=\prod_{(r,s)\in \lambda } q^rt^s\). Respectively,  for a vector space \(V\) with
\((\C^\times)^2\)-action  we denote \(e_{q,t}(V)\) is the Euler class of this space:
 \(e_{q,t}(V)=\prod_{(r,s)\in \mathrm{wt}(V)}(1-q^rt^s)\) where \(\mathrm{wt}(V)\) is the set of corresponding equivariant weights.
\begin{theorem}\label{thm:localization-f} For \(d_1,d_2>0\) 
  we have
  \[(qt)^{f(d_1,d_2,k)}\sum_{(\aaa,\bbb)\in \mathcal{P}} q^{|\aaa|}t^{|\bbb|}=\sum_{(\lambda,\mu)\in D(n,k)}\frac{c_{q,t}(\lambda)^{d_1}c_{q,t}(\mu)^{d_2}}{e_{q,t}(\Omega_{\lambda,\mu})}\]
  where \(\mathcal{P}=\mathcal{P}(d_1k+d_2(k+1),d_1+d_2)\) and \(f(d_1,d_2,k)=\binom{k+1}{3}d_1+\binom{k+2}{3}d_2.\)
\end{theorem}

\begin{proof}
This follows from Corollary \ref{cor: euler char} and equivariant localization. The additional factor $(qt)^{f(d_1,d_2,k)}$ is explained by Lemma \ref{lem: triangular correction} since the maps $\iota_k$ and $\iota_{k+1}$ twist the equivariant structure on the line bundle $\CO(1)$ by the characters of $\det\C[x,y]/(x,y)^k$ and $\det\C[x,y]/(x,y)^{k+1}$ respectively.
\end{proof}

\begin{remark}
  Numerical experiment suggests that the main formula in theorem~\ref{thm:localization-f} holds under weaker assumptions on \(d_1,d_2\).
  It appears that the formula is true whenever \(d_1+d_2\ge 0\) and \(d_1k+d_2 (k+1)\ge 0\). We do not have a geometric explanation for this
  observation.
\end{remark}


\begin{thebibliography}{99}

\bibitem{ABCH} D. Arcara, A. Bertram, I. Coskun, J. Huizenga. The minimal model program for the Hilbert scheme of points on $\mathbb{P}^2$ and Bridgeland stability. Adv. Math. (235): 580-626, 2013.

\bibitem{Bayer} A. Bayer, H. Chen, and Q. Jiang. Brill-Noether theory of Hilbert schemes of points on surfaces. Int.
Math. Res. Not. IMRN, (10):8403--8416, 2024.

\bibitem{BC} A. Bertram, and I. Coskun. The Birational Geometry of the Hilbert Scheme of Points on Surfaces. Birational Geometry, Rational Curves, and Arithmetic, Simons Symposia (2013).



\bibitem{FrobBook} Brion, M., and  Kumar, S. (2005). Frobenius splitting methods in
  geometry and representation theory.  Birkhäuser Boston.

  
\bibitem{CGM} E. Carlsson, E. Gorsky, A. Mellit. The 
$\mathbb{A}_{q,t}$ algebra and parabolic flag Hilbert schemes. Mathematische Annalen 376, no. 3 (2020): 1303--1336.

\bibitem{CGOT} I. Cavey, E. Gorsky, A. Oblomkov, J. Turner. Newton-Okounkov bodies for nested Hilbert schemes. arXiv:2510.07420

\bibitem{Fogarty} J. Fogarty, Algebraic families on an algebraic surface. Am. J. Math, Vol. 90
(1968), p. 511--521.

\bibitem{Fogarty2} J. Fogarty. Algebraic families on an algebraic surface II: the Picard scheme
of the punctual Hilbert scheme. Am. J. Math, Vol. 95 (1973), p. 660--687.

\bibitem{GGS} N. Gonzalez, E. Gorsky, J. Simental. Smooth correspondences between quiver varieties. arXiv:2601.22287

\bibitem{GKO}  E. Gorsky, O. Kivinen, A. Oblomkov. The affine Springer fiber -- sheaf correspondence.
Advances in Mathematics 464 (2025), article 110143.

\bibitem{Haiman} M. Haiman. Hilbert schemes, polygraphs and the Macdonald positivity conjecture.
J. Amer. Math. Soc. 14 (2001), no. 4, 941--1006.

\bibitem{Haimanqt}   M. Haiman. $t,q$-Catalan numbers and the Hilbert scheme.
Discrete Math. 193 (1998), no. 1--3, 201--224.

\bibitem{NY1} H. Nakajima, K. Yoshioka. Perverse coherent sheaves on blow-up. I. A quiver description. In Exploring new
structures and natural constructions in mathematical physics. Collected papers of the conference upon the occasion of
the retirement of Professor Akihiro Tsuchiya, Nagoya, Japan, March 5--8, 2007, pages 349--386. Tokyo: Mathematical
Society of Japan, 2011.

\bibitem{NY2} H. Nakajima, K. Yoshioka. Perverse coherent sheaves on blow-up. II: Wall-crossing and Betti numbers
formula. J. Algebr. Geom., 20(1):47--100, 2011.



\end{thebibliography}
\end{document}